\documentclass[american,12pt]{article}
\usepackage[T1]{fontenc}
\usepackage[utf8]{inputenc}
\usepackage{geometry}
\usepackage{array}
\usepackage{bm}
\usepackage{multirow}
\usepackage{amsmath}
\usepackage{amsthm}
\usepackage{amssymb}
\usepackage{setspace}
\usepackage{url}
\usepackage[authoryear]{natbib}
\usepackage{authblk}
\makeatletter

\providecommand{\tabularnewline}{\\}

\theoremstyle{plain}
\newtheorem{assumption}{\protect\assumptionname}
\theoremstyle{plain}
\newtheorem{thm}{\protect\theoremname}
\theoremstyle{plain}
\newtheorem{lem}{\protect\lemmaname}
\newenvironment{lyxlist}[1]
	{\begin{list}{}
		{\settowidth{\labelwidth}{#1}
		 \setlength{\leftmargin}{\labelwidth}
		 \addtolength{\leftmargin}{\labelsep}
		 }}
	{\end{list}}
\theoremstyle{plain}
\newtheorem{prop}{\protect\propositionname}

\makeatother

\usepackage{babel}
\providecommand{\assumptionname}{Assumption}
\providecommand{\lemmaname}{Lemma}
\providecommand{\propositionname}{Proposition}
\providecommand{\theoremname}{Theorem}

\begin{document}
\title{Variable selection via thresholding}

\author[1]{Ka Long Keith Ho}
\author[1,2]{Hien Duy Nguyen}
\affil[1]{Joint Graduate School of Mathematics for Innovation, Kyushu University, Fukuoka, Japan.}
\affil[2]{Department of Mathematics and Physical Science, La Trobe University, Melbourne, Australia.}

\maketitle
\begin{abstract}
Variable selection comprises an important step in many modern statistical
inference procedures. In the regression setting, when estimators cannot
shrink irrelevant signals to zero, covariates without relationships
to the response often manifest small but non-zero regression coefficients.
The ad hoc procedure of discarding variables whose coefficients are
smaller than some threshold is often employed in practice. We formally
analyze a version of such thresholding procedures and develop a simple
thresholding method that consistently estimates the set of relevant
variables under mild regularity assumptions. Using this thresholding
procedure, we propose a sparse, $\sqrt{n}$-consistent and asymptotically normal estimator whose non-zero elements do not exhibit shrinkage. The performance
and applicability of our approach are examined via numerical studies
of simulated and real data.
\end{abstract}


\section{Introduction}
Regression analysis is a fundamental tool of statistical inference.
In the modern context, the successful conduct of regression analysis
typically requires a solution to the variable selection problem, which
has been described and characterized in the works of \citet{george2000variable},
\citet{fan2010selective}, and \citet{heinze2018variable}, among
many others. 

Variable selection arises as a necessary task in all regression settings, including generalized linear models (GLMs), nonlinear regression, mixed-effects models, functional models, and so forth. However, due to the breadth of the subject, we will restrict our exposition to the linear regression setting, although we note that our methodology can be extended with minor modifications to most regression modeling tasks.

Let
\[
\bm{Y}=\mathbf{X}\bm{\beta}_{0}+\bm{\epsilon}\text{,}
\]
where $\bm{Y}=\left(Y_{1},\dots,Y_{n}\right)^{\top},\bm{\epsilon}=\left(\epsilon_{1},\dots,\epsilon_{n}\right)^{\top}\in\mathbb{R}^{n}$
are the response and error vectors, and $\mathbf{X}=\left(\bm{X}_{1},\dots,\bm{X}_{n}\right)^{\top}\in\mathbb{R}^{n\times p}$
is a design matrix, for $n,p\in\mathbb{N}$. Here, $\bm{\beta}_{0}=\left(\beta_{0,1},\dots,\beta_{0,p}\right)^{\top}\in\mathbb{R}^{p}$
is the vector of true signals which are unknown to the analyst, who
only observes $\mathbf{X}$ and $\bm{Y}$. We shall further restrict
ourselves to the random design scenario, where $\bm{X}_{i}$ and $\epsilon_{i}$
are random quantities that are independent and identically distributed
(IID) replicates of $\bm{X}$ and $\epsilon$, respectively, for each
$i\in\left\{ 1,\dots,n\right\} =\left[n\right]$. We denote the common
probability space on which all of our random objects are supported
on $\left(\Omega,{\cal F},\mathbb{P}\right)$ with expectation operator
$\mathbb{E}$.

We assume that $p$ is large but finite and that $\bm{\beta}_{0}$
is sparse in the sense that $\left|{\cal S}_{0}\right|\gg0$, where
\[
{\cal S}_{0}=\left\{ j\in\left[p\right]:\beta_{0,j}=0\right\} 
\]
is the set of irrelevant coordinates. The problem of variable selection
under sparsity can be rephrased as the problem of estimating the set
${\cal S}_{0}$. Traditional methods for addressing this problem,
such as step-wise and best subset selection methods (see e.g., \citealp[Sec. 3.3]{hastie2009elements}),
can be computationally prohibitive and slow, which has led to the large volume of research in sparse shrinkage estimators, such as the LASSO, bridge, and elastic net estimators, stemming from the works
of \citet{tibshirani1996regression}, \citet{fu1998penalized}, \citet{fan2001variable},
and \citet{zou2005regularization}, among other early pioneers. Detailed
accounts of the shrinkage estimation literature can be found in the
texts of \citet{buhlmann2011statistics}, \citet{hastie2015statistical},
\citet{rish2014sparse}, and \citet{van-de-Geer:2016aa}.

Typically, a shrinkage estimator $\hat{\bm{\beta}}_{n}=\left(\hat{\beta}_{n,1},\dots,\hat{\beta}_{n,p}\right)^{\top}$
will be such that the set of coordinates that are estimated to be
irrelevant,
\[
\hat{{\cal S}}_{0,n}=\left\{ j\in\left[p\right]:\hat{\beta}_{n,j}=0\right\} 
\]
is non-empty; i.e., $\left|\hat{{\cal S}}_{0,n}\right|>0$. However,
due to numerical implementations, it often occurs that the computed
elements of $\hat{\bm{\beta}}_{n}$ will yield small but non-zero
signals. This is the case, for example, when implementing the majorization--minimization
algorithms for sparse penalties, such as in \citet{Hunter:2005aa},
\citet{lloyd2018globally}, and \citet{abergel2024review}, or when
using the ``merit function'' penalties of \citet[Ch. 7]{zhao2018sparse}.
Furthermore, many shrinkage methods, such as the ridge estimator \citep{hoerl1970ridge}
and its adaptive variants \citep{frommlet2016adaptive,dai2018broken,abergel2024review}
 output elements of $\hat{\bm \beta}_{n}$ that are sufficiently
small as to be considered negligible. 

As per the references, an ad-hoc approach to handling such small estimates
is to simply declare some small threshold $\delta>0$ for which all
estimated signals with absolute size smaller than $\delta$ are placed
in the set of irrelevant coordinates. However, the question arises
as to how small should the threshold $\delta$ be chosen. Practically,
one may try a set of candidate thresholds and make a choice based
on some empirical criterion.

In the following text, we shall consider a formalization of the described thresholding scheme whereupon one estimates the set of irrelevant signals ${\cal S}_{0}$ by the set $\hat{{\cal S}}_{n}$ of coordinates corresponding to elements of the estimated signal $\hat{\bm{\beta}}_{n}$ whose size is less than some choice of a small threshold. That is, we produce an algorithm that takes some estimator $\hat{\bm{\beta}}_{n}$
of $\bm{\beta}_{0}$ as an input along with a sequence of threshold
values $\delta_{1}>\delta_{2}>\dots>\delta_{K}>0$, for some $K\in\mathbb{N}$.
The algorithm then outputs an estimate $\hat{{\cal S}}_{n}$ of the
set of irrelevant coordinates ${\cal S}_{0}$, by eliminating all
signals smaller than $\delta_{\hat{K}_n}$, where $\hat{K}_{n}\in\left[K\right]$
is chosen using an information criterion dependent on the sequence
of threshold values along with their corresponding thresholded empirical
risks. Under some regularity assumptions regarding the moments of
$\bm{X}$ and $\epsilon$, and the minimum threshold size $\delta_{K}$, we show that
our algorithm is consistent in the sense that 
\begin{equation}
\lim_{n\to\infty}\mathbb{P}\left(\hat{{\cal S}}_{n}=\mathcal{S}_{0}\right)=1\text{.}\label{eq:=000020consistent=000020S}
\end{equation}
To accompany our selection result, we also propose an estimator $\bar{\bm{\beta}}_n$, whose $j$th coordinate is zero, if $j\in\hat{\cal S}_n$, and equal to the $j$th coordinate of the initial estimator of the regression model $\hat{\bm\beta}_n$, otherwise (the so-called ``hard threshold'' estimator). Via the consistency of $\hat{\cal S}_n$, if the $\sqrt{n}$-blowup of $\hat{\bm\beta}_n$ converges in distribution, then we can show that $\bar{\bm{\beta}}_n$ satisfies an oracle result. Namely, this estimator is $\sqrt{n}$-consistent in its relevant coordinates, and the irrelevant signals are equal to zero with probability approaching one. Furthermore, if $\hat{\bm\beta}_n$ is asymptotically normal, then the relevant coordinates of $\bar{\bm{\beta}}_n$ are asymptotically normal as well.

Our technical results follow from recent work regarding the asymptotic
distribution of the minimum value of an empirical risk minimization
program by \citet{westerhout2024asymptotic} along with connections
between such asymptotics and information criteria, reported in \citet{nguyen2024panic}.
The proofs are further made possible using the random empirical processes
theory of \citet{van2007empirical}, as presented in \citet{van2023weak}.
Along with our theory, we also provide some illustrative numerical
studies to demonstrate our methodology.

We note that our approaches are different and complementary to those of previous works that seek to use thresholding for variable selection and screening. For instance, \cite{zhou2009thresholding}, \cite{van2011adaptive}, \cite{slawski2013non}, \cite{pokarowski2015combined}, \cite{zheng2014high}, and \cite{sun2019hard} consider thresholding already regularized estimators, in the high-dimensional setting. The regularized estimators are required to be well-behaved, via enforcement of strong moment and tail assumptions on the error $\epsilon$  and covariate $\bm{X}$ and eigenvalue assumptions on the Grammian, in order to establish variable selection consistency via finite sample bounds. In contrast, we take a wholly asymptotic approach, which provides qualitative instead of quantitative guarantees under  weaker and different assumptions. Furthermore, our asymptotic normality result is obtained under different demands than those available in the literature. For instance, with normal noise and diminishing thresholds, \cite{potscher2009distribution}, \cite{potscher2011distributional} and \cite{schneider2016confidence} obtain marginal asymptotic normality results for the hard-thresholded least squares estimators. This contrasts with our results, which guarantee multivariate asymptotic normality without making assumptions on the noise distribution and for thresholds strictly bounded away from zero.

The remainder of the manuscript proceeds as follows. In Section 2,
we describe the construction of our thresholding algorithm. Our main
theoretical results are presented in Section 3, along with required
technical preliminaries. Numerical studies are conducted in Section
4 and concluding remarks are provided in Section 5. Further proofs
and auxiliary results are provided in the Appendix.

\section{The thresholding algorithm}

We retain notation from the introduction, and for each $k\in\left[K\right]$,
we define a thresholding function $t_{k}:\mathbb{R}\to\left[0,1\right]$,
which serves as a smooth approximation to the step function threshold
at $\delta_{k}$:
\begin{equation}
\text{step}\left(b;\delta_{k}\right)=\begin{cases}
0 & \text{if }\left|b\right|\le\delta_{k}\text{,}\\
1 & \text{if }\left|b\right|>\delta_{k}\text{.}
\end{cases}\label{eq:=000020Step=000020function}
\end{equation}
We shall elaborate on the required properties of these functions in
the sequel.

Given our estimate $\hat{\bm{\beta}}_{n}$ of $\bm{\beta}_{0}$, we
can apply the $k$th thresholding transformation
\begin{equation}
\bm{T}_{k}\left(\bm{\beta};\hat{\bm{\beta}}_{n}\right)=\left(\beta_{1}t_{k}\left(\hat{\beta}_{n,1}\right),\dots,\beta_{p}t_{k}\left(\hat{\beta}_{n,p}\right)\right)^{\top}\text{,}\label{eq:=000020Thresholding=000020transformation}
\end{equation}
which intuitively sets elements of $\bm{\beta}$ whose corresponding
estimated signals are smaller than $\delta_{k}$ to zero. For brevity,
we will sometimes write (\ref{eq:=000020Thresholding=000020transformation})
as $\bm{T}_{n,k}\left(\bm{\beta}\right)$. Furthermore, with the notation
\[
\mathbf{D}_{k}\left(\bm{\beta}\right)=\text{diag}\left(t_{k}\left(\beta_{1}\right),\dots,t_{k}\left(\beta_{p}\right)\right)\text{,}
\]
the diagonal matrix whose diagonal comprises the elements of $t_{k}\left(\beta_{1}\right),\dots,t_{k}\left(\beta_{p}\right)$,
we further have the expression 
\[
\bm{T}_{n,k}\left(\bm{\beta}\right)=\mathbf{D}_{k}\left(\hat{\bm{\beta}}_{n}\right)\bm{\beta}\text{.}
\]

For each $k\in\left[K\right]$, we can now define the thresholded
empirical risk 
\[
R_{n,k}\left(\bm{\beta}\right)=\frac{1}{n}\sum_{i=1}^{n}\left\{ Y_{i}-\bm{X}_{i}^{\top}\bm{T}_{n,k}\left(\bm{\beta}\right)\right\} ^{2}\text{,}
\]
where $R_{n,k}:\mathbb{R}^{p}\to\mathbb{R}$ can be understood as
a random empirical process in the sense of \citet[Sec. 3.13]{van2023weak},
since it depends not only on the averaging process, but also on the
random vector $\hat{\bm{\beta}}_{n}$. If $\hat{\bm{\beta}}_{n}$
is a consistent estimator of $\bm{\beta}_{0}$, then we may anticipate
that $R_{n,k}$ will converge to some limiting risk functional $r_{k}$,
defined for each $\bm{\beta}\in\mathbb{R}^{p}$ by
\[
r_{k}\left(\bm{\beta}\right)=\mathbb{E}\left[\left\{ Y-\bm{X}^{\top}\bm{T}_{k}\left(\bm{\beta},\bm{\beta}_{0}\right)\right\} ^{2}\right]\text{,}
\]
whose minimum value we can write as
\[
\psi_{k}=\min_{\bm{\beta}\in\mathbb{R}^{p}}\ r_{k}\left(\bm{\beta}\right)\text{.}
\]
We note that this minimum and its empirical variant exist because
both $R_{n,k}$ and $r_{k}$ are quadratic, for each $k$ and $n$.
This is easier to see when writing:
\[
R_{n,k}\left(\bm{\beta}\right)=\frac{1}{n}\sum_{i=1}^{n}\left\{ Y_{i}-\bm{X}_{i}^{\top}\mathbf{D}_{k}\left(\hat{\bm{\beta}}_{n}\right)\bm{\beta}\right\} ^{2}\text{ and }r_{k}\left(\bm{\beta}\right)=\mathbb{E}\left[\left\{ Y-\bm{X}^{\top}\mathbf{D}_{k}\left(\beta_{0}\right)\bm{\beta}\right\} ^{2}\right]\text{.}
\]

Intuitively, the $j$th coordinate of $\bm{T}_{k}\left(\bm{\beta};\bm{\beta}_{0}\right)$
is $0$, when $\left|\beta_{0,j}\right| \leq \delta_{k}$ and $\bm{\beta}_{j}$,
otherwise. Thus, the minimum risk $\psi_{k}$ is exactly the expected
least squares value from fitting the linear regression model which
excludes all variables in the $k$th irrelevant set
\[
{\cal S}_{k}=\left\{ j\in\left[p\right]:\left|\beta_{0,j}\right|\le\delta_{k}\right\} \text{.}
\]
If ${\cal S}_{k}\subset{\cal S}_{0}$, then only truly irrelevant
coordinates are excluded when evaluating $r_{k}$, and thus $\psi_{k}$
is equal to
\[
\psi_{0}=\min_{\bm{\beta}\in\mathbb{R}^{p}}\frac{1}{n}\left\Vert \bm{Y}-\mathbf{X}^{\top}\bm{\beta}\right\Vert ^{2}=\frac{1}{n}\left\Vert \bm{Y}-\mathbf{X}^{\top}\bm{\beta}_{0}\right\Vert ^{2}\text{,}
\]
the mean squared error under the true signal $\bm{\beta}_{0}$. However,
if ${\cal S}_{0}\subset{\cal S}_{k}$, then some relevant coordinates
that are not in ${\cal S}_{0}$ may be excluded when computing $r_{k}$
and thus $\psi_{0}\le\psi_{k}$. Observe further that since $\delta_{1}>\dots>\delta_{K}$
is decreasing, we have ${\cal S}_{1}\supset\dots\supset{\cal S}_{K}$,
and as per the argument above, it holds that $\psi_{1}\ge\dots\ge\psi_{K}\ge\psi_{0}$.
The following assumption is taken to guarantee that ${\cal S}_{K}={\cal S}_{0}$,
and thus $\psi_{K}=\psi_{0}$.
\begin{assumption}
\label{assu:delta=000020beta=000020relationships}For each $j\in\left[p\right]$,
\[
\min_{j\notin{\cal S}_{0}}\left|\beta_{0,j}\right|>\delta_{K}\text{.}
\]
\end{assumption}
The assumption states that our choice of thresholds can distinguish between
no signal and the minimum signal. In the sparse regression literature,
this corresponds to the common beta-min condition, which bounds the
smallest signal away from zero (cf. \citealp{buhlmann2011statistics}). 

We can propose the ``variable selection'' problem of determining
\[
k_{0}=\min\left\{ \underset{k\in\left[K\right]}{\arg\min}\ \psi_{k}\right\} \text{.}
\]
That is, we wish to choose the largest threshold $\delta_{k_{0}}\in\left\{ \delta_{k}:k\in\left[K\right]\right\} $
for which ${\cal S}_{0}\subset{\cal S}_{k_{0}}$, thus yielding a
model that excludes the largest possible number of irrelevant variables
whilst achieving the minimum limiting risk value. 

To estimate $k_{0}$, we construct the information criterion-like
estimator 
\[
\hat{K}_{n}=\min\left\{ \underset{k\in\left[K\right]}{\arg\min}\ \left(\hat{\psi}_{n,k}+P_{n,k}\right)\right\} \text{.}
\]
Here, for each $k\in\left[K\right]$,
\[
\hat{\psi}_{n,k}=\min_{\bm{\beta}\in\mathbb{R}^{p}}\ R_{n,k}\left(\bm{\beta};\hat{\bm{\beta}}_{n}\right)
\]
is the minimum empirical risk under the $k$th thresholding transformation
and estimator of $\psi_{k}$, and $P_{n,k}:\Omega\times\Theta\to\mathbb{R}$
is its corresponding penalty function, which can be chosen so that
\begin{equation}
\lim_{n\to\infty}\mathbb{P}\left(\hat{K}_{n}=k_{0}\right)=1\text{.}\label{eq:=000020consistent=000020k*}
\end{equation}
For each $k$ and $n$, we can estimate ${\cal S}_{k}$ by
\[
{\cal S}_{k,n}=\left\{ j\in\left[p\right]:\left|\hat{\beta}_{n,j}\right|\leq \delta_{k}\right\} \text{.}
\]
With this notation, our information criterion then yields the estimate
$
\hat{{\cal S}}_{n}={\cal S}_{\hat{K}_{n},n}
$
of ${\cal S}_{0}$ which constitutes the output of our thresholding
algorithm and our primary object of interest. 

It follows from (\ref{eq:=000020consistent=000020k*}) that $\hat{{\cal S}}_{n}$
satisfies the variable screening property \citet[Sec. 2.7]{buhlmann2011statistics}:
\[
\lim_{n\to\infty}\mathbb{P}\left(\hat{{\cal S}}_{n}\subset{\cal S}_{0}\right)=1\text{.}
\]
In the sequel, we will demonstrate that $\hat{{\cal S}}_{n}$ is moreover
consistent in the sense of (\ref{eq:=000020consistent=000020S}).

\subsection{Thresholding functions}

Naively, thresholding of irrelevant coordinates from the signal vector
is conducted using the step function (\ref{eq:=000020Step=000020function}),
as often applied in practice. However, to prove the necessary limit
theorems, we require that the functions $t_{k}$ are sufficiently
smooth as per the following assumption.
\begin{assumption}
\label{assu:=000020thresholds}For each $k\in\left[K\right]$, the
thresholding function $t_{k}:\mathbb{R}\to\left[0,1\right]$ is twice
continuously differentiable and has the form $t_{k}=\tau_{k}\circ\left|\cdot\right|$,
where $\tau_{k}:\mathbb{R}_{\ge0}\to\left[0,1\right]$ is an increasing
function, such that $\tau_{k}\left(0\right)=0$ and for each $b>0$,
$\tau_{k}\left(b\right)>0$.
\end{assumption}
For any $h > 0 $,
define the cubic splines
\[
\tau_{k}\left(b\right)=\begin{cases}
0 & \text{if }b\le\delta_{k}\text{,}\\
\frac{4}{h^{3}}\left(x-\delta_{k}\right)^{3} & \text{if }\delta_{k}<b\le\delta_{k}+\frac{h}{2}\text{,}\\
\frac{4}{h^{3}}\left(x-\delta_{k}-h\right)^{3}+1 & \text{if }\delta_{k}+\frac{h}{2}<\delta_{k}+h\text{,}\\
1 & \text{if }b\ge\delta_{k}+h\text{.}
\end{cases}
\]
We can check that $t_{k}=\tau_{k}\circ\left|\cdot\right|$ satisfy
Assumption \ref{assu:=000020thresholds}, as required. For the remainder
of the manuscript, we will use this spline construction as our choice
for thresholding. However, as will be apparent in the sequel, the thresholding function serves as a purely technical device, because in practice, the thresholding procedure can be implemented directly with the step function. We shall elaborate on this point in Section \ref{subsec:Ordinary-least-squares=000020simulation}.

\subsection{Penalty functions}

We will require that the penalty functions satisfy the following assumption.
\begin{assumption}
\label{assu:=000020Penalty=000020shape}For each pair $k,l\in\left[K\right]$
, $P_{k,n}>0$ for all sufficiently large $n$, $P_{k,n}=o_{\mathbb{P}}\left(1\right)$,
and if $l>k$, then $\sqrt{n}\left(P_{l,n}-P_{k,n}\right)\overset{\mathbb{P}}{\longrightarrow}\infty$.
\end{assumption}
Let $\alpha:\mathbb{R}_{\ge0}\to\mathbb{R}_{\ge0}$ be a decreasing
function. Then, we can satisfy Assumption \ref{assu:=000020Penalty=000020shape}
by taking
\begin{equation}
P_{n,k}=\alpha\left(\delta_{k}\right)\frac{\log n}{\sqrt{n}}\text{.}\label{eq:=000020generic=000020pen}
\end{equation}
Penalties $P_{n,k}=O\left(n^{-1/2}\log n\right)$ were proposed in
\citet{sin1996information} and adopted in \citet{nguyen2024panic} and
\citet{westerhout2024asymptotic}, where they are referred to as Sin--White
information criteria (SWIC). Notice that SWIC differ from the more
commonly used BIC-type (Bayesian information criteria; \citealp{schwarz1978estimating})
penalties, which have the form $P_{n,k}=O\left(n^{-1}\log n\right)$.
Although BIC penalties provide milder penalization of the model complexity, they approach zero too quickly to distinguish between
the cases where ${\cal S}_{k}\subset{\cal S}_{0}$, unless stronger assumptions are made. However, as long as $P_{k,n}=o_{\mathbb{P}}\left(1\right)$,
which is satisfied by the BIC and AIC-type (Akaike information criterion;
\citealp{akaike1974new}), we are guaranteed at least to have the variable screening
property for $\hat{{\cal S}}_{n}$.

We note that there is some room to experiment with the functional
form of $\alpha$, which we will investigate in our numerical studies.
However, a principled approach can be taken, such as via some analogy of the ``slope
heuristic'' of \citet{birge2007minimal}, as demonstrated in \citet{baudry2012slope}.

\section{Theoretical results}

The first step of our analysis is to establish that $\sqrt{n}\left(\hat{\psi}_{n,k}-\psi_{k}\right)$
converges weakly, for each $k\in\left[K\right]$. To facilitate our
technical exposition, we require the following notation. Let $\Theta\subset\mathbb{R}^{p}$
be a compact subset, which we can choose to be arbitrarily large as
to cause no practical restriction. For each $k\in\left[K\right]$
, $i\in\left[n\right]$, and $\bm{\beta},\bm{\beta}^{\prime}\in\Theta$, we
denote the individual level losses by
\[
l_{\bm{\beta},\bm{\beta}^{\prime}}\left(\bm{X}_{i},Y_{i}\right)=\left\{ Y_{i}-\bm{X}_{i}^{\top}\bm{T}_{k}\left(\bm{\beta};\bm{\beta}^{\prime}\right)\right\} ^{2}\text{,}
\]
where $l_{\bm{\beta},\bm{\beta}^{\prime}}\left(\bm{X}_{i},Y_{i}\right)$
has the same law as $l_{\bm{\beta},\bm{\beta}^{\prime}}\left(\bm{X},Y\right)$.
We then define the empirical and expected thresholded risk functions by
\[
G_{n,k}\left(\bm{\beta};\bm{\beta}^{\prime}\right)=\frac{1}{n}\sum_{i=1}^{n}l_{\bm{\beta},\bm{\beta}^{\prime}}\left(\bm{X}_{i},Y_{i}\right)\text{ and }g_{k}\left(\bm{\beta};\bm{\beta}^{\prime}\right)=\mathbb{E}\left[l_{\bm{\beta},\bm{\beta}^{\prime}}\left(\bm{X},Y\right)\right]\text{,}
\]
which allows us to write
\[
R_{n,k}=G_{n,k}\left(\cdot;\hat{\bm{\beta}}_{n}\right)\text{ and }r_{k}=g_{k}\left(\cdot;\bm{\beta}_{0}\right)\text{.}
\]
Since our analysis is the same for each fixed $k$, we will also find
it convenient to remove the index and write $G_{n}$ and $g$ in place
of $G_{n,k}$ and $g_{k}$, and $R_{n}$ and $r$ in place of $R_{n,k}$
and $r$, respectively, for example.

Consider that the thresholding function $\bm{T}\left(\bm{\beta};\bm{\beta}_{0}\right)=\mathbf{D}\left(\bm{\beta}_{0}\right)\bm{\beta}$
partitions the coordinates $j\in\left[p\right]$ into the set of irrelevant
coordinates ${\cal S}$ and its complement, the set of relevant coordinates
${\cal S}^{\text{c}}$. Let us write the parameter $\bm{\beta}$ and
design vectors restricted to the relevant coordinates as $\bm{\beta}_{\text{R}}$
and $\bm{X}_{\text{R}}$, respectively, along with the relevant thresholded
expected risk
\[
r_{\text{R}}\left(\bm{\beta}_{\text{R}}\right)=\mathbb{E}\left[\left\{ Y-\mathbf{X}_{R}\mathbf{D}_{\text{R}}\left(\bm{\beta}_{0}\right)\bm{\beta}_{\text{R}}\right\} ^{2}\right]\text{,}
\]
where $\mathbf{D}_{\text{R}}\left(\bm{\beta}_{0}\right)$ is the same
as $\mathbf{D}\left(\bm{\beta}_{0}\right)$ with empty rows and columns
removed.

Observe that we can consider $R_{n}:\Theta\to\mathbb{R}$ and $r:\Theta\to\mathbb{R}$
to be elements of the set $\ell^{\infty}\left(\Theta\right)$ of bounded
functionals supported on $\Theta$ (cf. \citealp[Sec. 1.5]{van2023weak}).
As such, we start by showing that $\sqrt{n}\left(R_{n}-r\right)$
converges weakly in $\ell^{\infty}\left(\Theta\right)$, as $n\to\infty$.
Denoting weak convergence by $\rightsquigarrow$, we make the following
assumptions.
\begin{assumption}
\label{assu:fourth=000020moments}The variables $\bm{X}$ and $\epsilon$
have finite fourth moments: i.e., $\mathbb{E}\left[\left\Vert \bm{X}\right\Vert ^{4}\right]<\infty$
and $\mathbb{E}\left[\epsilon^{4}\right]<\infty$.
\end{assumption}
\begin{assumption}
\label{assu:asymptotic=000020beta}The estimator $\hat{\bm{\beta}}_{n}$
converges in distribution to $\bm{\beta}_{0}$: i.e., 
\[
\sqrt{n}\left(\hat{\bm{\beta}}_{n}-\bm{\beta}_{0}\right)\rightsquigarrow\bm{Z}\text{,}
\]
for some random vector $\bm{Z}:\Omega\to\mathbb{R}^{p}$.
\end{assumption}
Assumption \ref{assu:fourth=000020moments} is a moment condition required
for our empirical process weak convergence results, and Assumption
\ref{assu:asymptotic=000020beta} is satisfied, for example, whenever
$\hat{\bm{\beta}}_{n}$ is an asymptotically normal estimator, as
in the case of the ordinary least squares (OLS) estimator:
\begin{equation}
\hat{\bm{\beta}}_{n}=\left(\mathbf{X}^{\top}\mathbf{X}\right)^{-1}\mathbf{X}^{\top}\bm{Y}\text{.}\label{eq:=000020OLS}
\end{equation}

Let $\bm{\beta}\mapsto\nabla_{2}g\left(\bm{\beta};\bm{\beta}_0\right)$ be the gradient
of $g\left(\bm{\beta};\bm{\beta}^{\prime}\right)$ in $\bm{\beta}^{\prime}$, evaluated at $\bm{\beta}^{\prime}=\bm{\beta}_{0}$,
as a function of the $\bm{\beta}$. The following theorem is technical
but forms the foundation of our subsequent analysis.
\begin{thm}
\label{thm:random=000020empirical=000020process=000020weak}Under
Assumptions \ref{assu:=000020thresholds}, \ref{assu:fourth=000020moments}
and \ref{assu:asymptotic=000020beta}, 
\[
\sqrt{n}\left(R_{n}-r\right)\rightsquigarrow R+\bm{Z}^{\top}\nabla_{2}g(\cdot,\bm{\beta}_0)\text{,}
\]
where $R:\Omega\to\ell^{\infty}\left(\Theta\right)$ is a Gaussian
process.
\end{thm}
The proof of Theorem \ref{thm:random=000020empirical=000020process=000020weak}
follows from the theory of random empirical processes and is provided
in Appendix \ref{sec:Proof=000020of=000020Theorem=0000201}. Intuitively,
the theorem states that for each $k$, when centered by the limiting
risk and thresholded by the true signal $\bm{\beta}_{0}$, the empirical thresholded
least squares converges to a random process in $\ell^{\infty}\left(\Theta\right)$.

Let $\iota:\ell^{\infty}\left(\Theta\right)\to\mathbb{R}$ denote
the infimum function, whereby for each $h:\Theta\to\mathbb{R}$,
\[
\iota\left(h\right)=\inf_{\bm{\beta}\in\Theta}h\left(\bm{\beta}\right)\text{.}
\]
With this notation, we can also define the $\varepsilon$-minimizers
of $h$ as the set
\[
{\cal B}_{\varepsilon}\left(h\right)=\left\{ \bm{\beta}\in\Theta:h\left(\bm{\beta}\right)\le\iota\left(h\right)+\varepsilon\right\} \text{.}
\]
\citet[Thm. 3.1]{westerhout2024asymptotic} then states that the directional
derivative of $\iota$ at $h$ in the direction of $\eta\in\ell^{\infty}\left(\Theta\right)$
is
\[
\dot{\iota}_{h}\left(\eta\right)=\lim_{\varepsilon\searrow0}\inf_{\bm{\beta}\in{\cal B}_{\varepsilon}\left(h\right)}\eta\left(\bm{\beta}\right)\text{.}
\]
We can apply the directional functional delta method of \citet{Romisch:2014aa}
(see also \citealp[Fact 3.2]{westerhout2024asymptotic}) to obtain
the following result.
\begin{lem}
\label{lem:weak=000020convergence=000020min}Under the assumptions
of Theorem \ref{thm:random=000020empirical=000020process=000020weak},
we can choose a compact $\Theta\subset\mathbb{R}^{p}$ such that 
\[
\sqrt{n}\left(\iota\left(R_{n}\right)-\psi\right)\rightsquigarrow F\text{,}
\]
where $F=\dot{\iota}_{r}\left(R+\bm{Z}^{\top}\nabla_{2}g\right)$
is a random variable.
\end{lem}
\begin{proof}
The result follows directly from the directional functional delta
method to show that $\sqrt{n}\left(\iota\left(R_{n}\right)-\iota\left(r\right)\right)\rightsquigarrow F$.
It then suffices to check that we can choose a large enough $\Theta$
so that $\iota\left(r\right)=\psi$. Indeed, $r$ is quadratic and
is minimized on $\mathbb{R}^{p}$ at any $\bm{\beta}_{}^{*}$ that
solves the Fermat condition
\[
\mathbf{D}_{}\left(\bm{\beta}_{0}\right)\mathbb{E}\left[\bm{X}Y\right]=\mathbf{D}\left(\bm{\beta}_{0}\right)\mathbb{E}\left[\bm{X}\bm{X}^{\top}\right]\mathbf{D}\left(\bm{\beta}_{0}\right)\bm{\beta}_{}^{*}\text{.}
\]
The system has potentially infinitely many solutions $\bm{\beta}_{}^{*}$, although
the relevant coordinates of $\bm{\beta}_{}^{*}$ will be the same
as the solution
\[
\bm{\beta}_{\text{R}}^{*}=\left\{ \mathbf{D}_{\text{R}}\left(\bm{\beta}_{0}\right)\mathbb{E}\left[\bm{X}_{\text{R}}\bm{X}_{\text{R}}^{\top}\right]\mathbf{D}_{\text{R}}\left(\bm{\beta}_{0}\right)\right\} ^{-1}\mathbf{D}_{\text{R}}\left(\bm{\beta}_{0}\right)\mathbb{E}\left[\bm{X}_{\text{R}}Y\right]\text{,}
\]
the minimizer of $r_{\text{R}}$ on $\mathbb{R}^{\left|{\cal S}^{\text{c}}\right|}$.
In particular, $r$ is minimized at $\bm{\beta}^{**}$ which equals
$\bm{\beta}_{\text{R}}^{*}$ on the relevant coordinates and $0$
elsewhere. We thus have $\bm{\beta}^{**}\in\left[-m,m\right]^{p}\subset\Theta$,
where $m=\max_{j\in\left[p\right]}\left|\beta_{j}^{**}\right|$ .
\end{proof}
Note that since $K$ is finite, a compact $\Theta$ large
enough so that Lemma \ref{lem:weak=000020convergence=000020min} holds
for all $k\in\left[K\right]$ can be constructed as a union over the
choice for each $k$. As a consequence of Lemma \ref{lem:weak=000020convergence=000020min},
the same proof as that of \citet[Thm. 1]{nguyen2024panic} yields
the consistency of $\hat{K}_{n}$ and thus the variable screening
property of $\hat{{\cal S}}_{n}$.
\begin{thm}
\label{thm:Consistenct=000020of=000020K}Under Assumptions \ref{assu:=000020thresholds}--\ref{assu:asymptotic=000020beta},
we can choose a compact $\Theta\subset\mathbb{R}^{p}$ such that
\[
\lim_{n\to\infty}\mathbb{P}\left(\hat{K}_{n}=k_{0}\right)=1\text{.}
\]
\end{thm}
The following consistency theorem comprises our main technical result
and is a direct consequence of the consistency of $\hat{K}_{n}$.
\begin{thm}
\label{thm:=000020Consistent=000020S}Under Assumptions \ref{assu:delta=000020beta=000020relationships}--\ref{assu:asymptotic=000020beta},
we can choose a compact $\Theta\subset\mathbb{R}^{p}$ such that 
\[
\lim_{n\to\infty}\mathbb{P}\left(\hat{{\cal S}}_{n}={\cal S}_{0}\right)=1\text{.}
\]
\end{thm}
\begin{proof}
We first show that ${\cal S}_{k^{0}}={\cal S}_{0}$. Under Assumption
\ref{assu:delta=000020beta=000020relationships}, $\min_{j\notin{\cal S}_{0}}\left|\beta_{0,j}\right|>\delta_{K}$,
we have that the least squares loss 
\[
\tilde{r}_{}\left(\bm{\beta}\right)=\mathbb{E}\left[\left\{ Y-\bm{X}^{\top}\bm{\beta}\right\} ^{2}\right]
\]
has the same minimum as that of the relevant loss function
\[
\tilde{r}_{\text{R}}\left(\bm{\beta}_{\text{R}}\right)=\mathbb{E}\left[\left\{ Y-\bm{X}_{\text{R}}^{\top}\bm{\beta}_{\text{R}}\right\} ^{2}\right]\text{,}
\]
where the relevant set is ${\cal S}_{0}^{\text{c}}={\cal S}_{K}^{\text{c}}$.
However, we note that that this minimum is also the same as that of
the relevant thresholded loss $r_{\text{R}}=r_{K,\text{R}}$ since,
for each $j\in{\cal S}_{0}^{\text{c}}$, $\beta_{j}\mapsto\beta_{j}t_{K}\left(\beta_{0,j}\right)$
is surjective because $t_{K}\left(\beta_{0,j}\right)$ is positive
by Assumption \ref{assu:delta=000020beta=000020relationships} (cf.
\citealp[Sec. A.4.11]{hinderer2016dynamic}). If we write the minimizer
of $\tilde{r}_{\text{R}}$ as $\tilde{\bm{\beta}}_{\text{R}}^{*}$,
i.e., the least square solution
\[
\tilde{\bm{\beta}}_{\text{R}}^{*}=\mathbb{E}\left[\bm{X}_{\text{R}}\bm{X}_{\text{R}}^{\top}\right]^{-1}\mathbb{E}\left[\bm{X}_{\text{R}}Y\right]\text{,}
\]
then the minimizer of $r_{\text{R}}$ is simply a scaling of $\tilde{\bm{\beta}}_{\text{R}}^{*}$,
of the form
\[
\bm{\beta}_{\text{R}}^{*}=\left[\tilde{\beta}_{\text{R},j}^{*}/t_{K}\left(\beta_{0,j}\right)\right]_{j\in{\cal S}_{0}^{\text{c}}}\text{.}
\]
In the same way, for each $k\in\left[K\right]$, we can argue that
$k$th relevant thresholded loss $r_{k\text{,}\text{R}}$ has the
same minimum as that of $k$th relevant least squares loss
\[
\tilde{r}_{k,\text{R}}\left(\bm{\beta}_{k,\text{R}}\right)=\mathbb{E}\left[\left\{ Y-\bm{X}_{k,\text{R}}^{\top}\bm{\beta}_{k,\text{R}}\right\} ^{2}\right]\text{,}
\]
defined by the relevant set ${\cal S}_{k}^{\text{c}}=\left\{ j\in\left[p\right]:\left|\beta_{0,j}\right|>\delta_{k}\right\} $.
Thus, with the minimizer of $\tilde{r}_{k,\text{R}}$ written as $\tilde{\bm{\beta}}_{k,\text{R}}^{*}$,
we can obtain the minimizer $\bm{\beta}_{k,\text{R}}^{*}$ of $r_{k,\text{R}}$
via the same scaling argument. Finally, since $r_{k}$ has the same
minimum as $r_{k,\text{R}}$, where we can take the minimizer $\bm{\beta}_{k}^{*}$
to have relevant elements equal to $\bm{\beta}_{k,\text{R}}^{*}$
and irrelevant elements $0$. But since $t_{k}\left(\beta_{0,j}\right)>0$,
for each $k\in\left[K\right]$ and $j\in{\cal S}_{0}^{\text{c}}$,
we can construct a sufficiently large $\Theta\supset\left[-m,m\right]^{p}$
such that $\bm{\beta}_{k}^{*}\in\Theta$ for all $k$, by choosing
\[
m=\max\left\{ \left|\beta_{k,j}^{*}\right|/t_{k}\left(\beta_{0,j}\right):j\in\left[p\right],k\in\left[K\right]\right\} \text{,}
\]
with the convention that $0/0=0$. We have thus shown that $\iota\left(r_{k}\right)=\psi_{k}$
for each $k\in\left[K\right]$. 

Due to Assumption \ref{assu:delta=000020beta=000020relationships},
we have that $\psi_{k}\ge\psi_{K}=\psi_{0}$, with equality if and
only if ${\cal S}_{k}\subset{\cal S}_{K}={\cal S}_{0}$. But by definition
of $k_{0}$, it must hold that $\psi_{k_{0}}=\psi_{K}$, and thus
${\cal S}_{k_{0}}\subset{\cal S}_{0}$. The fact that ${\cal S}_{0}\subset{\cal S}_{k_{0}}$
is trivial by observation that $t_{k}\left(0\right)=0$ for each $k$,
under Assumption \ref{assu:=000020thresholds}.

Next, observe that
\begin{align*}
\mathbb{P}\left(\hat{{\cal S}}_{n}={\cal S}_{0}\right)=\mathbb{P}\left({\cal S}_{\hat{K}_{n},n}={\cal S}_{k_{0}}\right)\ge\mathbb{P}\left(\hat{K}_{n}=k_{0},{\cal S}_{k_{0},n}={\cal S}_{k_{0}}\right)\text{,}
\end{align*}
and by Bonferonni's inequality, 
$$\mathbb{P}\left(\hat{K}_{n}=k_{0},{\cal S}_{k_{0},n}={\cal S}_{k_{0}}\right)\ge1-\mathbb{P}\left(\hat{K}_{n}\ne k_{0}\right)-\mathbb{P}\left({\cal S}_{k_{0},n}\ne{\cal S}_{k_{0}}\right).$$
Since $\mathbb{P}\left(\hat{K}_{n}=k_{0}\right)\to1$ by Theorem \ref{thm:Consistenct=000020of=000020K},
it suffices to show that $\mathbb{P}\left({\cal {\cal S}}_{k_{0},n}={\cal S}_{k_{0}}\right)\to1$.
Writing $\mathbf{1}\left(\mathsf{A}\right)$ as the indicator function
of the statement $\mathsf{A}$, we have that
\begin{align*}
\mathbb{P}\left({\cal {\cal S}}_{k_{0},n}\ne{\cal S}_{k_{0}}\right) & =\mathbb{P}\left(\bigcup_{j=1}^{p}\left\{ \mathbf{1}\left(t_{k_{0}}\left(\hat{\beta}_{n,j}\right)=0\right)\ne\mathbf{1}\left(t_{k_{0}}\left(\beta_{0,j}\right)=0\right)\right\} \right)\\
 & \le\sum_{j=1}^{p}\mathbb{P}\left(\mathbf{1}\left(t_{k_{0}}\left(\hat{\beta}_{n,j}\right)=0\right)\ne\mathbf{1}\left(t_{k_{0}}\left(\beta_{0,j}\right)=0\right)\right)\\
 & \le\sum_{j=1}^{p}\mathbb{P}\left(\left|\left|\hat{\beta}_{n,j}\right|-\left|\beta_{0,j}\right|\right|>c\right)\text{,}
\end{align*}
where $c=\min_{j\in\left[p\right]}\left|\delta_{k_{0}}-\left|\beta_{0,j}\right|\right|$.
But by Assumption \ref{assu:asymptotic=000020beta}, $\hat{\bm{\beta}}_{n}\stackrel{\mathbb{P}}{\longrightarrow}\bm{\beta}_{0}$
and it follows from continuous mapping that 
\[
\lim_{n\to\infty}\sum_{j=1}^{p}\mathbb{P}\left(\left|\left|\hat{\beta}_{n,j}\right|-\left|\beta_{0,j}\right|\right|>c\right)=0\text{,}
\]
which completes the proof.
\end{proof}
Finally, we return to one of our initial motivations: to justify the ad-hoc procedure of thresholding out ``small'' estimates of $\hat{\bm\beta}_n$. Using the estimate $\hat K_n$, we define the (hard) thresholded estimator $\bar{\bm\beta}_n$ component-wise by $\bar{\beta}_{n,j} = \hat{\beta}_{n,j} \mathbf{1}\left(t_{\hat K_{n}}(\hat{\beta}_{nj}) > 0\right)$, for each $j\in[p]$. This is equivalent to setting the $j$th coordinate of $\bar{\bm\beta}_n$ to zero if $j\in\hat{\cal S}_n$ and declaring the elements of all other coordinates equal to the corresponding elements of $\hat{\bm\beta}_n$. Below, we demonstrate a so-called oracle property for $\bar{\bm\beta}_n$, by combining the convergence in distribution of $\hat{\bm{\beta}}_n$ and the consistency of $\hat{\cal S}_n$. Using the true irrelevant set ${\cal S}_{0}$, we can write $\bm{\beta}_0 = (\bm{\beta}_{0,\mathrm R},\bm{\beta}_{0,\mathrm I}) = (\bm{\beta}_{0,\mathrm R},\bm{0})$, where $\bm{\beta}_{\mathrm I}$ denotes the subvector of $\bm{\beta}$ that consists of only the irrelevant coordinates. The same notation extends to other vectors of length $p$; e.g., ${\bm Z}_\mathrm{R}$ and ${\bm Z}_\mathrm{I}$ denote the relevant and irrelevant subvectors of ${\bm Z}$, respectively. The following result is proved in Appendix \ref{sec:Proof=000020of=000020theorem=0000204}.

\begin{thm}\label{thmoracle}
Under Assumptions \ref{assu:delta=000020beta=000020relationships}--\ref{assu:asymptotic=000020beta},
$\sqrt{n}(\bar{\bm{\beta}}_{n,\mathrm R} - \bm{\beta}_{0,\mathrm R}) \rightsquigarrow \bm{Z}_{\mathrm R}$ and $\lim_{n\to\infty}\mathbb{P}(\bar{\bm{\beta}}_{n,\mathrm I} = \bm{0}) = 1$.
\end{thm}

The result of Theorem \ref{thmoracle} therefore justifies that the thresholding process is a one-step improvement to any non-sparse estimator that converges in distribution, producing a sparse estimator that has the same asymptotic distribution in the relevant coordinates. Furthermore, if $\bm{Z}$ is multivariate normal then so is $\bm{Z}_{\mathrm R}$, and thus $\bar{\bm{\beta}}_{n,\mathrm R}$ is an asymptotically normal estimator of $\bm{\beta}_{0,\mathrm R}$. Observe that this result provides the same oracle guarantees as those for sparse penalization estimators under similar conditions, such as those considered in \cite{fan2001variable}.

\section{Numerical studies}

All numerical results were produced in the R programming environment
and all codes are made available at \url{https://github.com/karfho/Threshold_IC}.

\subsection{Simulation studies}
\subsubsection{Ordinary least squares}\label{subsec:Ordinary-least-squares=000020simulation}

We begin by considering the performance of the thresholding procedure
in the linear regression setting, with various sample sizes and covariate
dimensions. Namely, we consider $n\in\left\{ 100,1000,10000\right\} $
and $p\in\left\{ 20,50\right\} $, where we replicated each simulation
scenario 100 times. For each $n$ and $p$, we simulate covariates
$\left(\bm{X}_{i}\right)_{i\in\left[n\right]}$ IID with the same
distribution as $\bm{X}\sim\text{N}\left(0,\bm{\Sigma}\right)$, and
we simulate errors $\left(\epsilon_{i}\right)_{i\in\left[n\right]}$
IID with standard normal distribution. The covariate matrix is taken
to have $\text{1}$s along the diagonal and $1/5$ in all off-diagonal
elements. This setup is representative of mildly correlated covariate
scenarios.

We generate $\left(Y_{i}\right)_{i\in\left[n\right]}$ as $Y_{i}=\bm{X}_{i}^{\top}\bm{\beta}_{0}+\epsilon_{i}$,
where we consider two scenarios for $\bm{\beta}_{0}$:
\begin{lyxlist}{00.00.0000}
\item [{S1}] $\bm{\beta}_{0}=\left(0.2,0.4,\dots,1.8,2.0,0,\dots,0\right)$,
\item [{S2}] $\bm{\beta}_{0}=\left(0.05,0.1,\dots,0.45,0.5,0,\dots,0\right)$,
\end{lyxlist}
with S1 and S2 corresponding to high and low signal-to-noise ratio
scenarios, respectively. Note that in both cases, the set of relevant
coordinates is ${\cal S}_{0}^{\text{c}}=\left[10\right]$, which is
independent of both $p$ and $n$. For each simulation replication,
we estimate $\bm{\beta}_{0}$ by the OLS (\ref{eq:=000020OLS}): $\hat{\bm{\beta}}_{n}$,
set $K=p$, and use threshold values
\begin{equation}
\left(\delta_{k}\right)_{k\in\left[K\right]}=\left(\left|\hat{\beta}_{n,\left(j\right)}\right|\right)_{i\in\left[p\right]}\text{,}\label{eq:=000020empirical=000020threshold}
\end{equation}
where $\left(\hat{\beta}_{n,\left(j\right)}\right)_{j\in\left[p\right]}$
are the elements of the estimator $\hat{\bm{\beta}}_{n}$, ordered
from largest to smallest: i.e., $\left|\hat{\beta}_{n,\left(1\right)}\right|<\left|\hat{\beta}_{n,\left(2\right)}\right|\dots<\left|\hat{\beta}_{n,\left(p\right)}\right|$.
We note that although our theory characterizes the behavior under a deterministic sequence $\left(\delta_{k}\right)_{k\in\left[K\right]}$, the choice is sensible because the procedure will threshold the signals at exactly where the increase in maximum signal magnitude of the relevant set occurs, which allows us to consider all possible leveled sets with minimal computation. However, if we operate with a predetermined set of thresholding values, then it is possible to miss out on certain leveled sets if multiple components of $\left|\hat\beta_n\right|$ lie in between some $[\delta_k, \delta_{k+1})$, and if no component falls in between $[\delta_k, \delta_{k+1})$, recomputing the minimal empirical risk is simply redundant. 

To satisfy Assumption \ref{assu:=000020Penalty=000020shape}, we noted
that one can take penalties of SWIC form: (\ref{eq:=000020generic=000020pen}).
For our purposes, we will consider a further restriction to the case
where
\begin{equation}
\alpha\left(\delta_{k}\right)=\frac{c}{\delta_{k}^{r}}\text{,}\label{eq:=000020SWIC=000020alpha}
\end{equation}
for some choices of $c,r>0$. In particular, to represent weak, medium,
and strong penalties, respectively, we take $\left(c,r\right)=\left(0.5,0.25\right)$,
$\left(0.75,0.4\right)$, and $\left(1,0.5\right)$. 

To measure the performance of our method, in the face of our theoretical
result that $\hat{K}_{n}$ exhibits model selection consistency, we
report the average (over replications) estimated threshold value $\delta_{\hat{K}_{n}}$,
average percentage of covariates that are relevant that are incorrectly
estimated to be irrelevant, and the average percentage of irrelevant
covariates that are correctly estimated to be irrelevant. The two
latter indices can be characterized as the average false negative
rate (FNR) and average true negative rate (TNR), respectively, as
percentages where
\[
\text{FNR}=1-\frac{\left|\hat{{\cal S}}_{n}^{\text{c}}\cap{\cal S}_{0}^{\text{c}}\right|}{\left|{\cal S}_{0}^{\text{c}}\right|}\text{, and }\text{TNR}=\frac{\left|\hat{{\cal S}}_{n}\cap{\cal S}_{0}\right|}{\left|{\cal S}_{0}\right|}\text{.}
\]
We report the outcomes in Tables \ref{tab:S1=000020OLS} and \ref{tab:S2=000020OLS}.

\begin{table}
\caption{Results for S1 using OLS estimators. All averages are computed over
100 replications. Note that $\delta_{k_0}<0.2$.}\label{tab:S1=000020OLS}

\centering{}%
\begin{tabular}{|c|cc|ccc|}
\hline 
\multicolumn{1}{|c}{} &  & \multicolumn{1}{c}{} & \multicolumn{3}{c|}{$\left(c,r\right)$}\tabularnewline
\multicolumn{1}{|c}{Measure} & $n$ & \multicolumn{1}{c}{$p$} & $\left(0.5,0.25\right)$ & $\left(0.75,0.4\right)$ & $\left(1,0.5\right)$\tabularnewline
\hline 
\hline 
\multirow{6}{*}{$\delta_{\hat{K}_{n}}$} & 100 & 20 & 0.0995 & 0.1726 & 0.2181\tabularnewline
 & 1000 & 20 & 0.06332 & 0.09462 & 0.1606\tabularnewline
 & 10000 & 20 & 0.0207 & 0.0207 & 0.0207\tabularnewline
 & 100 & 50 & 0.07655 & 0.1217 & 0.1729\tabularnewline
 & 1000 & 50 & 0.06087 & 0.121 & 0.156\tabularnewline
 & 10000 & 50 & 0.02739 & 0.02739 & 0.02739\tabularnewline
\hline 
\multirow{6}{*}{$\text{FNR}\times100\%$} & 100 & 20 & 2.8 & 6.3 & 10.2\tabularnewline
 & 1000 & 20 & 0.3 & 3.3 & 7.8\tabularnewline
 & 10000 & 20 & 0 & 0 & 0\tabularnewline
 & 100 & 50 & 2.1 & 3.3 & 4.1\tabularnewline
 & 1000 & 50 & 0.2 & 4.2 & 7\tabularnewline
 & 10000 & 50 & 0 & 0 & 0\tabularnewline
\hline 
\multirow{6}{*}{$\text{TNR}\times100\%$} & 100 & 20 & 64.8 & 87 & 96.4\tabularnewline
 & 1000 & 20 & 99.7 & 100 & 100\tabularnewline
 & 10000 & 20 & 100 & 100 & 100\tabularnewline
 & 100 & 50 & 38.075 & 55.85 & 70.425\tabularnewline
 & 1000 & 50 & 84.75 & 100 & 100\tabularnewline
 & 10000 & 50 & 100 & 100 & 100\tabularnewline
\hline 
\end{tabular}
\end{table}

\begin{table}
\caption{Results for S2 using OLS estimators. All averages are computed over
100 replications. Note that $\delta_{k_0}<0.05$.}\label{tab:S2=000020OLS}

\centering{}%
\begin{tabular}{|c|cc|ccc|}
\hline 
\multicolumn{1}{|c}{} &  & \multicolumn{1}{c}{} & \multicolumn{3}{c|}{$\left(c,r\right)$}\tabularnewline
\multicolumn{1}{|c}{Measure} & $n$ & \multicolumn{1}{c}{$p$} & $\left(0.5,0.25\right)$ & $\left(0.75,0.4\right)$ & $\left(1,0.5\right)$\tabularnewline
\hline 
\hline 
\multirow{6}{*}{$\delta_{\hat{K}_{n}}$} & 100 & 20 & 0.08472 & 0.1413 & 0.17\tabularnewline
 & 1000 & 20 & 0.0624 & 0.09443 & 0.12\tabularnewline
 & 10000 & 20 & 0.04549 & 0.0676 & 0.08767\tabularnewline
 & 100 & 50 & 0.07323 & 0.1143 & 0.1474\tabularnewline
 & 1000 & 50 & 0.04186 & 0.09816 & 0.1262\tabularnewline
 & 10000 & 50 & 0.04349 & 0.06519 & 0.09206\tabularnewline
\hline 
\multirow{6}{*}{$\text{FNR}\times100\%$} & 100 & 20 & 16.9 & 27.7 & 34.2\tabularnewline
 & 1000 & 20 & 11 & 20.1 & 25.3\tabularnewline
 & 10000 & 20 & 8.8 & 14.3 & 18\tabularnewline
 & 100 & 50 & 11.4 & 19 & 23.4\tabularnewline
 & 1000 & 50 & 5.4 & 18.8 & 27.3\tabularnewline
 & 10000 & 50 & 8.2 & 13.6 & 18.8\tabularnewline
\hline 
\multirow{6}{*}{$\text{TNR}\times100\%$} & 100 & 20 & 56.2 & 78.3 & 85.2\tabularnewline
 & 1000 & 20 & 96.2 & 100 & 100\tabularnewline
 & 10000 & 20 & 100 & 100 & 100\tabularnewline
 & 100 & 50 & 36.85 & 53.5 & 64.7\tabularnewline
 & 1000 & 50 & 74.325 & 99.675 & 100\tabularnewline
 & 10000 & 50 & 100 & 100 & 100\tabularnewline
\hline 
\end{tabular}
\end{table}

The results of Tables \ref{tab:S1=000020OLS} and \ref{tab:S2=000020OLS}
appear to indicate that the thresholding method usefully performs
variable selection, especially for larger values of $n$. As expected,
we observe that weaker penalization yields more conservative thresholding, and thus lowers the FNR, although at the expense of decreased TNR, and vice versa for stronger penalization. Supporting the asymptotic theory, we observe that both $(1 -$ TNR$)$ and FNR decrease as $n$ increases, and similarly, the optimal selected thresholds $\delta_{\hat{K}_{n}}$
get closer to the upper bound of $\delta_{k_0}$. Comparing Table
\ref{tab:S2=000020OLS} to Table \ref{tab:S1=000020OLS}, we observe,
as expected that the overall performance of the thresholding method
declines when the signal is weaker, and thus for comparable performance,
one requires larger sample sizes. Interestingly, we note that the
performance in TNR does not decline as severely as that for FNR. With
better calibration of the penalty function $\alpha$, one may anticipate
better overall performance in each scenario, although we know of no
general and theoretically supported approach for such a task in our setting.

\subsubsection{Ridge and adaptive ridge estimators}\label{subsec:Ridge-and-adaptive}

As an alternative to OLS, we may consider instead taking $\hat{\bm{\beta}}_{n}$
to be a ridge regression-type estimator, which naturally benefits
from the use of a thresholding scheme. A noted shortcoming of $l_{2}$
regularization approaches, such as ridge regression and adaptive ridge
regression, is that the shrinkage estimators are not sparse, like
the LASSO, but instead merely shrink weak signals arbitrarily close
to zero. As such, many practitioners have taken to simply conduct
arbitrary thresholding of small signals as an ad hoc approach to variable
selection, which our thresholding approach now makes rigorous. 

In this section, we consider the same simulation as in Section \ref{subsec:Ordinary-least-squares=000020simulation},
but instead using the classic ridge estimator \citep{hoerl1970ridge}
and the adaptive ridge (AR) estimator of \citet{dai2018broken} and
\citet{Ho:2024aa}. In particular, for tuning parameters $\lambda_{n},\xi_{n}>0$,
we define the ridge estimator as
\[
\hat{\bm{\beta}}_{\text{Ridge},n}=\left(\mathbf{X}^{\top}\mathbf{X}+\lambda_{n}\mathbf{I}\right)^{-1}\mathbf{X}^{\top}\bm{Y}\text{,}
\]
and for $s\in\mathbb{N}$, we define the AR estimator at the $s$th
step by
\[
\hat{\bm{\beta}}_{\text{AR},n}^{\left(s\right)}=\left(\mathbf{X}^{\top}\mathbf{X}+\xi_{n}\bm{\mathbf{\mathfrak{D}}}\left(\hat{\bm{\beta}}_{\text{AR},n}^{\left(s-1\right)}\right)\right)^{-1}\mathbf{X}^{\top}\bm{Y}\text{,}
\]
where $\bm{\mathbf{\mathfrak{D}}}\left(\bm{\beta}\right)=\text{diag}\left(1/\beta_{1}^{2},\dots,1/\beta_{p}^{2}\right)$,
and $\hat{\bm{\beta}}_{\text{AR},n}^{\left(0\right)}=\hat{\bm{\beta}}_{\text{Ridge},n}$.
Using $\lambda_{n}=\sqrt{n}$ and $\xi_{n}=1$, we compare the ridge
and AR estimator, with $s=5$, to the OLS for S1 when $n=100$ and
S2 when $n=1000$, with results reported in Tables \ref{tab:Ridge=000020S1}
and \ref{tab:Ridge=000020S2}, respectively.

\begin{table}
\caption{Comparison of performance between OLS, ridge, and AR estimators for
scenario S1 with $n=100$. Note that $\delta_{k_0}<0.2$ and $s=5$.}\label{tab:Ridge=000020S1}

\centering{}%
\begin{tabular}{|cc|ccc|}
\hline 
 & \multicolumn{1}{c}{} & \multicolumn{3}{c|}{$\left(c,r\right)$}\tabularnewline
Method & \multicolumn{1}{c}{Measure} & $\left(0.5,0.25\right)$ & $\left(0.75,0.4\right)$ & $\left(1,0.5\right)$\tabularnewline
\hline 
\hline 
\multirow{3}{*}{OLS} & $\delta_{\hat{K}_{n}}$ & 0.0995 & 0.1726 & 0.2181\tabularnewline
 & $\text{FNR}\times100\%$ & 2.8 & 6.3 & 10.2\tabularnewline
 & $\text{TNR}\times100\%$ & 64.8 & 87 & 96.4\tabularnewline
\hline 
\multirow{3}{*}{Ridge} & $\delta_{\hat{K}_{n}}$ & 0.1213 & 0.1962 & 0.2335\tabularnewline
 & $\text{FNR}\times100\%$ & 2.7 & 5.9 & 9.1\tabularnewline
 & $\text{TNR}\times100\%$ & 69.2 & 91 & 96.8\tabularnewline
\hline 
\multirow{3}{*}{AR} & $\delta_{\hat{K}_{n}}$ & 0.04553 & 0.09792 & 0.1506\tabularnewline
 & $\text{FNR}\times100\%$ & 6 & 8.9 & 11.9\tabularnewline
 & $\text{TNR}\times100\%$ & 94.2 & 97.2 & 98.7\tabularnewline
\hline 
\end{tabular}
\end{table}

\begin{table}
\caption{Comparison of performance between OLS, ridge, and AR estimators for
scenario S2 with $n=1000$. Note that $\delta_{k_0}<0.05$ and $s=5$.}\label{tab:Ridge=000020S2}

\centering{}%
\begin{tabular}{|cc|ccc|}
\hline 
 & \multicolumn{1}{c}{} & \multicolumn{3}{c|}{$\left(c,r\right)$}\tabularnewline
Method & \multicolumn{1}{c}{Measure} & $\left(0.5,0.25\right)$ & $\left(0.75,0.4\right)$ & $\left(1,0.5\right)$\tabularnewline
\hline 
\hline 
\multirow{3}{*}{OLS} & $\delta_{\hat{K}_{n}}$ & 0.04186 & 0.09816 & 0.1262\tabularnewline
 & $\text{FNR}\times100\%$ & 5.4 & 18.8 & 27.3\tabularnewline
 & $\text{TNR}\times100\%$ & 74.325 & 99.675 & 100\tabularnewline
\hline 
\multirow{3}{*}{Ridge} & $\delta_{\hat{K}_{n}}$ & 0.04006 & 0.09563 & 0.1225\tabularnewline
 & $\text{FNR}\times100\%$ & 5 & 18.9 & 27.3\tabularnewline
 & $\text{TNR}\times100\%$ & 74.05 & 99.725 & 100\tabularnewline
\hline 
\multirow{3}{*}{AR} & $\delta_{\hat{K}_{n}}$ & 0.05081 & 0.09292 & 0.1174\tabularnewline
 & $\text{FNR}\times100\%$ & 10.5 & 21.8 & 27.7\tabularnewline
 & $\text{TNR}\times100\%$ & 96.825 & 99.975 & 100\tabularnewline
\hline 
\end{tabular}
\end{table}

From Tables \ref{tab:Ridge=000020S1} and \ref{tab:Ridge=000020S2},
we observe that the OLS and ridge-type estimators perform similarly,
although there appear to be slight improvements when using the ridge-type
estimators. We note that since ridge and AR estimators conduct shrinkage,
more signals tend to be eliminated under each threshold, and thus
the thresholding algorithm is is better at eliminating irrelevant
signals. This can be seen in the improvement of TNR in both tables.
In particular, we observe that the AR estimator conducts very aggressive
shrinkage, causing much higher TNR than both OLS and ridge estimators,
even when the penalty is small. 

\subsection{Prostate cancer data example}

To illustrate the application of our method in practice, we consider
the prostate data set of \citet{stamey1989prostate}. These data consist
of $n=97$ observations, each with covariates consisting of $p=8$
clinical measurements from individual patients: the logarithm of cancer
volume $\left(X_{1}\right)$, logarithm of weight $\left(X_{2}\right)$,
age $\left(X_{3}\right)$, logarithm of the amount of benign prostatic
hyperplasia $\left(X_{4}\right)$, seminal vesicle invasion $\left(X_{5}\right)$,
logarithm of capsular penetration $\left(X_{6}\right)$, Gleason score
$\left(X_{7}\right)$, and percentage Gleason scores of four or five
$\left(X_{8}\right)$. The response variable is the logarithm of prostate-specific
antigen $\left(Y\right)$.

Following typical preprocessing protocols, we standardize the data
and, with $K=8$, we apply our thresholding method with $\left(\delta_{k}\right)_{k\in\left[K\right]}$
chosen as per (\ref{eq:=000020empirical=000020threshold}), using
SWIC penalties with $\alpha$ functions (\ref{eq:=000020SWIC=000020alpha}),
taking $\left(c,r\right)=\left(0.5,0.25\right)$, $\left(0.75,0.4\right)$,
and $\left(1,0.5\right)$, as in the simulations. The OLS, ridge and
AR estimators are used with $\lambda_{n}=\sqrt{n}$, $\xi_{n}=1$,
and $s=5$, as per Section \ref{subsec:Ridge-and-adaptive}. We report
the sets of relevant covariates obtained via the thresholding algorithm
${\cal \hat{S}}_{n}^{\text{c}}$ in Table \ref{tab:Prostate=000020nointer}.
Further, expanding the set of covariates by also considering pairwise
interactions, yielding a total of $p=36$ predictors, we report the
sizes of the sets of relevant covariates obtained from thresholding
$\left|\hat{{\cal S}}_{n}^{\text{c}}\right|$ in Table \ref{tab:Prostate=000020inter}.

\begin{table}
\caption{Sets of relevant coordinates $\hat{{\cal S}}_{n}$ obtained from
thresholding of prostate cancer data using OLS, ridge, and AR estimators.}\label{tab:Prostate=000020nointer}

\centering{}%
\begin{tabular}{|cccc|}
\hline 
 & \multicolumn{3}{c|}{$\left(c,r\right)$}\tabularnewline
Method & $\left(0.5,0.25\right)$ & $\left(0.75,0.4\right)$ & $\left(1,0.5\right)$\tabularnewline
\hline 
\hline 
\multirow{1}{*}{OLS} & $\left\{ 1,2,3,4,5,6,8\right\} $ & $\left\{ 1,2,5\right\} $ & $\left\{ 1\right\} $\tabularnewline
\hline 
\multirow{1}{*}{Ridge} & $\left\{ 1,2,5\right\} $ & $\left\{ 1,2,5\right\} $ & $\left\{ 1\right\} $\tabularnewline
\hline 
\multirow{1}{*}{AR} & $\left\{ 1,2,5\right\} $ & $\left\{ 1\right\} $ & $\left\{ 1\right\} $\tabularnewline
\hline 
\end{tabular}
\end{table}

\begin{table}
\caption{Size of relevant covariate sets $\hat{{\cal S}}_{n}$ obtained from
thresholding of prostate cancer data using OLS, ridge, and AR estimators,
with pairwise interactions.}\label{tab:Prostate=000020inter}

\centering{}%
\begin{tabular}{|cccc|}
\hline 
 & \multicolumn{3}{c|}{$\left(c,r\right)$}\tabularnewline
Method & $\left(0.5,0.25\right)$ & $\left(0.75,0.4\right)$ & $\left(1,0.5\right)$\tabularnewline
\hline 
\hline 
\multirow{1}{*}{OLS} & 15 & 14 & 14\tabularnewline
\hline 
\multirow{1}{*}{Ridge} & 30 & 14 & 2\tabularnewline
\hline 
\multirow{1}{*}{AR} & 3 & 1 & 1\tabularnewline
\hline 
\end{tabular}
\end{table}

From both Tables \ref{tab:Prostate=000020nointer} and \ref{tab:Prostate=000020inter},
we observe as expected, that stronger penalization implies less variables
being declared relevant, regardless of initial estimator used, while
in all but the first column of Table \ref{tab:Prostate=000020inter},
we observe that the hierarchy of shrinkage from OLS, ridge, and AR
estimators yields decreasing numbers of relevant covariates, as anticipated.
From Table \ref{tab:Prostate=000020nointer}, we conclude that the
important variables are $X_{1},X_{2},X_{5}$, with $X_{1}$ appearing
to be the most relevant of all covariates. This corroborates with
the outcomes from \cite{zou2005regularization}, \cite{lin2010bayesian} and \cite{sun2013consistent}, which identify the same variables as being most important.

\section{Conclusion}

Variable selection is a crucial step in many modern inference workflows,
with shrinkage estimators forming a useful class of methods for identifying
important covariates in the sparse signal setting. We have introduced
a systematic approach that takes any initial regression estimator
that exhibits convergence in distribution and outputs a consistent
estimator of the set of relevant covariates, under practical regularity
conditions. Our numerical studies demonstrate that our approach performs as theoretically anticipated, and functions well using OLS, ridge,
and AR estimators used as inputs, although our theory suggests that
the class of input estimators whose performance are guaranteed by
theory is far larger.

We anticipate three immediate avenues for future research. First,
to permit applications to infinite dimensional settings, we need to
allow for $K$ to be infinitely large, either countably or uncountably.
It is know that in some nested regression settings, model selection
procedures with uncountably large $K$ are possible (see, e.g., \citealp{zhang2024consistent}),
although the situations are quite different and we do not anticipate
the proof techniques to transfer without additional technical development.
Secondly, we have developed our theory in the linear regression setting
for ease of exposition and to keep the scope manageable. However,
with necessary modifications, there is little else that inhibits the
application of our approach to the generalized linear model, nonlinear
regression, linear functional regression, or other parametric regression
modeling settings where the notion of variable selection applies.
Thirdly, our current theory makes few restrictions on the class of
allowable penalty functions, and in effect, makes no recommendations
regarding the design of an efficient penalty function. Similar questions
of optimal penalization appears in the setting of model selection
oracle bounds, manifesting as the so called slope heuristic omnibus
procedure (cf. \citealp{birge2007minimal}, \citealp{baudry2012slope},
and \citealp{arlot2019minimal}). Although the breadth of application
of the slope heuristic does not apply in our asymptotic setting, it
is possible that an analogous procedure for automatic penalty calibration
exists.

\section*{Acknowledgments} 
KLKH is supported by the GPMI program at Kyushu University. HDN is funded by Australian Research Council grants: DP230100905 and DP250100860.

\appendix

\part*{Appendix}

\section{Proof of Theorem 1}\label{sec:Proof=000020of=000020Theorem=0000201}

Following the decomposition of \citet[Sec. 3.13]{van2023weak}, we
begin by writing:
\begin{eqnarray*}
\sqrt{n}\left\{ G_{n}\left(\bm{\beta};\hat{\bm{\beta}}_{n}\right)-g\left(\bm{\beta};\bm{\beta}_{0}\right)\right\}  & = & \sqrt{n}\left[\left\{ G_{n}\left(\bm{\beta};\hat{\bm{\beta}}_{n}\right)-g\left(\bm{\beta};\hat{\bm{\beta}}_{n}\right)\right\} -\left\{ G_{n}\left(\bm{\beta};\bm{\beta}_{0}\right)-g\left(\bm{\beta};\bm{\beta}_{0}\right)\right\} \right]\\
 &  & +\sqrt{n}\left\{ G_{n}\left(\bm{\beta};\bm{\beta}_{0}\right)-g\left(\bm{\beta};\bm{\beta}_{0}\right)\right\} +\sqrt{n}\left\{ g\left(\bm{\beta};\hat{\bm{\beta}}_{n}\right)-g\left(\bm{\beta};\bm{\beta}_{0}\right)\right\} \\
 & = & \text{(1A)}+\text{(1B)}+\text{(1C)}\text{.}
\end{eqnarray*}

Our program is to prove the weak convergence of terms (1B) and (1C)
to $R$ and $\bm{Z}^{\top}\nabla_{2}g(\cdot,\bm{\beta}_0)$, respectively, and to demonstrate
that $\text{(1A)}=o_{\mathbb{P}}\left(1\right)$. The following result
follows from \citet[Thm. 3.13.4]{van2023weak}.
\begin{prop}
\label{prop:1A}Under Assumptions \ref{assu:=000020thresholds}, \ref{assu:fourth=000020moments},
and \ref{assu:asymptotic=000020beta}, it holds that
\[
\sup_{\bm{\beta}\in\Theta}\left|\sqrt{n}\left[\left\{ G_{n}\left(\bm{\beta};\hat{\bm{\beta}}_{n}\right)-g\left(\bm{\beta};\hat{\bm{\beta}}_{n}\right)\right\} -\left\{ G_{n}\left(\bm{\beta};\bm{\beta}_{0}\right)-g\left(\bm{\beta};\bm{\beta}_{0}\right)\right\} \right]\right|\stackrel{\mathbb{P}}{\longrightarrow}0\text{.}
\]
\end{prop}
By \citet[Thm. 3.13.4]{van2023weak}, it suffices to check that the
set 
\[
{\cal L}=\left\{ l_{\bm{\beta},\bm{\beta}^{\prime}}\left(\bm{X},Y\right)=\left\{ Y-\bm{X}^{\top}\bm{T}_{k}\left(\bm{\beta};\bm{\beta}^{\prime}\right)\right\} ^{2}:\left(\bm{\beta},\bm{\beta}^{\prime}\right)\in\Theta\times\Theta\right\} 
\]
is $\mathbb{P}$-Donsker and it holds that
\[
\bar{\Delta}_{n}=\sup_{\bm{\beta}\in\Theta}P\left[\left\{ l_{\bm{\beta},\hat{\bm{\beta}}_{n}}\left(\bm{X},Y\right)-l_{\bm{\beta},\bm{\beta}_{0}}\left(\bm{X},Y\right)\right\} ^{2}\right]\stackrel{\mathbb{P}}{\longrightarrow}0\text{,}
\]
where we use $P$ to denote the expectation only over $\left(\bm{X},Y\right)$
and not $\hat{\bm{\beta}}_{n}$. To verify that ${\cal L}$ is $\mathbb{P}$-Donsker,
we apply \citet[Prop. 9.79]{shapiro2021lectures} which requires that
we check the conditions that (i) $\mathbb{E}\left[l_{\bm{\beta}_{1},\bm{\beta}_{1}^{\prime}}\left(\bm{X},Y\right)^{2}\right]<\infty$
and (ii) there exists a function $\bar{L}:\mathbb{R}^{p+1}\to\mathbb{R}$
with $\mathbb{E}\left[\bar{L}\left(\bm{X},Y\right)^{2}\right]<\infty$,
where 
\[
\left|l_{\bm{\beta}_{1},\bm{\beta}_{1}^{\prime}}\left(\bm{X},Y\right)-l_{\bm{\beta}_{2},\bm{\beta}_{2}^{\prime}}\left(\bm{X},Y\right)\right|\le\bar{L}\left(\bm{X},Y\right)\left[\left\Vert \bm{\beta}_{1}-\bm{\beta}_{2}\right\Vert +\left\Vert \bm{\beta}_{1}^{\prime}-\bm{\beta}_{2}^{\prime}\right\Vert \right]\text{,}
\]
for every $\bm{\beta}_{1},\bm{\beta}_{2},\bm{\beta}_{1}^{\prime},\bm{\beta}_{2}^{\prime}\in\Theta$.

To check (i), we start by writing
\begin{align*}
l_{\bm{\beta},\bm{\beta}^{\prime}}\left(\bm{X},Y\right) & =\left\{ Y-\bm{X}^{\top}\bm{T}_{k}\left(\bm{\beta};\bm{\beta}^{\prime}\right)\right\} ^{2}\\
 & =\left\{ \bm{X}^{\top}\left[\bm{\beta}_{0}-\bm{T}_{k}\left(\bm{\beta};\bm{\beta}^{\prime}\right)\right]+\epsilon\right\} ^{2}\\
 & =\left\{ \sum_{j=1}^{p}X_{j}\beta_{0j}-\sum_{j=1}^{p}X_{j}t_{k}\left(\beta_{j}^{\prime}\right)\beta_{j}+\epsilon\right\} ^{2}
\end{align*}
and thus
\begin{align*}
\mathbb{E}\left[l_{\bm{\beta}_{1},\bm{\beta}_{1}^{\prime}}\left(\bm{X},Y\right)^{2}\right] & =\mathbb{E}\left[\left\{ \bm{X}^{\top}\left[\bm{\beta}_{0}-\bm{T}_{k}\left(\bm{\beta}_{1};\bm{\beta}_{1}^{\prime}\right)\right]+\epsilon\right\} ^{4}\right]\\
 & \le\left(1+p\right)\left\{ \mathbb{E}\left[\left|\epsilon\right|^{4}\right]+\sum_{j=1}^{p}\left(\left|\beta_{0j}\right|+t_{k}\left(\beta_{1j}^{\prime}\right)\left|\beta_{1j}\right|\right)^{4}\mathbb{E}\left[\left|X_{j}\right|^{4}\right]\right\} \text{,}\\
 & \le\left(1+p\right)\left\{ \mathbb{E}\left[\left|\epsilon\right|^{4}\right]+16B^{4}\sum_{j=1}^{p}\mathbb{E}\left[\left|X_{j}\right|^{4}\right]\right\} \text{,}
\end{align*}
where $B=\sup_{\bm{\beta}\in\Theta}\left\Vert \bm{\beta}\right\Vert _{1}$,
since $0\le t_{k}\le1$. Thus, under the moment condition of Assumption
\ref{assu:fourth=000020moments}, (i) is verified. Next, observe that
the derivative of $l$ in $\left(\bm{\beta},\bm{\beta}^{\prime}\right)$
is
\[
\nabla l_{\bm{\beta},\bm{\beta}^{\prime}}\left(\bm{X},Y\right)=-2\left(\sum_{j=1}^{p}X_{j}\beta_{0j}-\sum_{j=1}^{p}X_{j}t_{k}\left(\beta_{j}^{\prime}\right)\beta_{j}+\epsilon\right)\left(\dots,X_{j}{t_{k}}\left(\beta_{j}^{\prime}\right),\dots,X_{j}\beta_{j}\dot{t_{k}}\left(\beta_{j}^{\prime}\right),\dots\right)^{\top}\text{,}
\]
where $\dot{t_{k}}$ is the derivative of $t_{k}$, which exists and
is continuous due to Assumption \ref{assu:=000020thresholds}. Thus,
we can write
\begin{equation}
\left\Vert \nabla l_{\bm{\beta},\bm{\beta}^{\prime}}\left(\bm{X},Y\right)\right\Vert _{1}\le C\sum_{j=1}^{p}\sum_{j^{\prime}=1}^{p}a_{j,j^{\prime}}\left(\bm{\beta},\bm{\beta}^{\prime}\right)\left|X_{j}X_{j^{\prime}}\right|+\sum_{j=1}^{p}b_{j}\left(\bm{\beta},\bm{\beta}^{\prime}\right)\left|X_{j}\epsilon\right|\text{,}\label{eq:gradient=000020expansion}
\end{equation}
for some large constant $C<\infty$, and continuous functions $a_{j,j^{\prime}},b_{j}:\Theta\times\Theta\to\mathbb{R}_{\ge0}$.
Since $\Theta$ is compact, we can therefore bound (\ref{eq:gradient=000020expansion})
by the supremum over $\Theta\times\Theta$ which exists and is finite
by the Weierstrass extreme value theorem. Thus, there is a large constant
$C<\infty$ such that
\[
\left\Vert \nabla l_{\bm{\beta},\bm{\beta}^{\prime}}\left(\bm{X},Y\right)\right\Vert _{1}\le C\left\{ \sum_{j=1}^{p}\sum_{j^{\prime}=1}^{p}\left|X_{j}X_{j^{\prime}}\right|+\sum_{j=1}^{p}\left|X_{j}\epsilon\right|\right\} \text{.}
\]
But by the mean-value theorem, since $l_{\bm{\beta},\bm{\beta}^{\prime}}\left(Y,\bm{X}\right)$
is continuously differentiable and $\Theta$ is compact, we can write
\[
\left|l_{\bm{\beta}_{1},\bm{\beta}_{1}^{\prime}}\left(\bm{X},Y\right)-l_{\bm{\beta}_{2},\bm{\beta}_{2}^{\prime}}\left(\bm{X},Y\right)\right|\le\bar{L}\left(\bm{X},Y\right)\left[\left\Vert \bm{\beta}_{1}-\bm{\beta}_{2}\right\Vert +\left\Vert \bm{\beta}_{1}^{\prime}-\bm{\beta}_{2}^{\prime}\right\Vert \right]\text{,}
\]
for each $\bm{\beta}_{1},\bm{\beta}_{2},\bm{\beta}_{1}^{\prime},\bm{\beta}_{2}^{\prime}\in\Theta$,
where $\bar{L}\left(\bm{X},Y\right)=C\left\Vert \nabla l_{\bm{\beta},\bm{\beta}^{\prime}}\left(\bm{X},Y\right)\right\Vert _{1}$,
for some large $C<\infty$, and therefore we can take 
\[
\bar{L}\left(\bm{X},Y\right)=C\left\{ \sum_{j=1}^{p}\sum_{j^{\prime}=1}^{p}\left|X_{j}X_{j^{\prime}}\right|+\sum_{j=1}^{p}\left|X_{j}\epsilon\right|\right\} \text{.}
\]
Clearly, we have that for a constant $C<\infty$,
\[
\bar{L}\left(\bm{X},Y\right)^{2}\le C\left\{ \sum_{j=1}^{p}\sum_{j^{\prime}=1}^{p}\left|X_{j}X_{j^{\prime}}\right|^{2}+\sum_{j=1}^{p}\left|X_{j}\epsilon\right|^{2}\right\} \text{.}
\]
Assumption \ref{assu:fourth=000020moments} then implies that 
\[
\mathbb{E}\left[\bar{L}\left(\bm{X},Y\right)^{2}\right]\le C\left\{ \sum_{j=1}^{p}\sum_{j^{\prime}=1}^{p}\mathbb{E}\left[\left|X_{j}X_{j^{\prime}}\right|^{2}\right]+\sum_{j=1}^{p}\left[\mathbb{E}\left|X_{j}\epsilon\right|^{2}\right]\right\} <\infty\text{,}
\]
and thus (ii) holds, implying that ${\cal L}$ is $\mathbb{P}$-Donsker.

Next, to check that $\bar{\Delta}_{n}\overset{\mathbb{P}}{\longrightarrow}0$,
we define
\begin{align*}
\Delta_{n}\left(\bm{\beta}\right) & =P\left[\left\{ l_{\bm{\beta},\hat{\bm{\beta}}_{n}}\left(\bm{X},Y\right)-l_{\bm{\beta},\bm{\beta}_{0}}\left(\bm{X},Y\right)\right\} ^{2}\right]\text{.}
\end{align*}
We require pair of results from \citet[Thms. 22.9 and 22.10]{davidson2021stochastic},
which we specify in the context of our problem.
\begin{lem}
\label{lem:If=000020SE=000020then=000020uniform}If $\Delta_{n}\left(\bm{\beta}\right)\overset{\mathbb{P}}{\longrightarrow}0$
for each $\bm{\beta}\in\Theta$ and $\left(\Delta_{n}\right)$ is
stochastic equicontinuous, then $\bar{\Delta}_{n}\overset{\mathbb{P}}{\longrightarrow}0$.
\end{lem}
\begin{lem}
\label{lem:Equicont}For sufficiently large $n\in\mathbb{N}$, if
\[
\left|\Delta_{n}\left(\bm{\beta}\right)-\Delta_{n}\left(\bm{\beta}^{\prime}\right)\right|\le\bar{D}_{n}\left\Vert \bm{\beta}-\bm{\beta}^{\prime}\right\Vert \text{,}
\]
for each $\bm{\beta},\bm{\beta}^{\prime}\in\Theta$, where $\bar{D}_{n}=O_{\mathbb{P}}\left(1\right)$,
then $\left(\Delta_{n}\right)$ is stochastic equicontinuous.
\end{lem}
We begin by checking the first condition of Lemma \ref{lem:If=000020SE=000020then=000020uniform}.
Let us write:
\[
\Delta_{n}\left(\bm{\beta}\right)=P\left[l_{\bm{\beta},\hat{\bm{\beta}}_{n}}\left(\bm{X},Y\right)^{2}\right]-2P\left[l_{\bm{\beta},\hat{\bm{\beta}}_{n}}\left(\bm{X},Y\right)l_{\bm{\beta},\bm{\beta}_{0}}\left(\bm{X},Y\right)\right]+P\left[l_{\bm{\beta},\bm{\beta}_{0}}\left(\bm{X},Y\right)^{2}\right]\text{.}
\]
But since $P\left[l_{\bm{\beta},\bm{\beta}^{\prime}}\left(\bm{X},Y\right)^{2}\right]=\mathbb{E}\left[l_{\bm{\beta},\bm{\beta}^{\prime}}\left(\bm{X},Y\right)^{2}\right]<\infty$
uniformly in $\left(\bm{\beta},\bm{\beta}^{\prime}\right)$, as already
checked, we have that $P\left[l_{\bm{\beta},\bm{\beta}^{\prime}}\left(\bm{X},Y\right)^{2}\right]$
is a continuous function in $\left(\bm{\beta},\bm{\beta}^{\prime}\right)$,
by \citet[Thm. 9.55]{shapiro2021lectures}. We can similarly check
that the middle term is is continuous. Since $\hat{\bm{\beta}}_{n}$
is a consistent estimator by Assumption \ref{assu:asymptotic=000020beta},
we have that 
\[
\Delta_{n}\left(\bm{\beta}\right)\overset{\mathbb{P}}{\longrightarrow}P\left[l_{\bm{\beta},\bm{\beta}_{0}}\left(\bm{X},Y\right)^{2}\right]-2P\left[l_{\bm{\beta},\bm{\beta}_{0}}\left(\bm{X},Y\right)l_{\bm{\beta},\bm{\beta}_{0}}\left(\bm{X},Y\right)\right]+P\left[l_{\bm{\beta},\bm{\beta}_{0}}\left(\bm{X},Y\right)^{2}\right]=0\text{,}
\]
for each $\bm{\beta}$, as required. To apply Lemma \ref{lem:Equicont},
we observe that we can write
\[
\left|\Delta_{n}\left(\bm{\beta}\right)-\Delta_{n}\left(\bm{\beta}^{\prime}\right)\right|\le\text{(2A)}+\text{(2B)}+\text{(2C)}\text{,}
\]
where
\[
\text{(2A)}=\left|P\left[l_{\bm{\beta},\hat{\bm{\beta}}_{n}}\left(\bm{X},Y\right)^{2}\right]-P\left[l_{\bm{\beta}^{\prime},\hat{\bm{\beta}}_{n}}\left(\bm{X},Y\right)^{2}\right]\right|\text{,}
\]
\[
\text{(2B)}=2\left|P\left[l_{\bm{\beta},\hat{\bm{\beta}}_{n}}\left(\bm{X},Y\right)l_{\bm{\beta}^{\prime},\bm{\beta}_{0}}\left(\bm{X},Y\right)\right]-P\left[l_{\bm{\beta}^{\prime},\hat{\bm{\beta}}_{n}}\left(\bm{X},Y\right)l_{\bm{\beta}^{\prime},\bm{\beta}_{0}}\left(\bm{X},Y\right)\right]\right|\text{,}
\]
\[
\text{(2C)}=\left|P\left[l_{\bm{\beta},\bm{\beta}_{0}}\left(\bm{X},Y\right)^{2}\right]-P\left[l_{\bm{\beta}^{\prime},\bm{\beta}_{0}}\left(\bm{X},Y\right)^{2}\right]\right|\text{.}
\]
We can thus verify the condition of Lemma \ref{lem:Equicont} for
each component. We will do so using the mean value theorem. In the
case of (2A) and (2C), this requires that $P\left[l_{\bm{\beta},\bm{\beta}^{\prime}}\left(\bm{X},Y\right)^{2}\right]$
is differentiable in $\bm{\beta}$, which holds if it satisfies the
Lipschitz condition:
\[
\left|l_{\bm{\beta}_{1},\bm{\beta}^{\prime}}\left(\bm{X},Y\right)^{2}-l_{\bm{\beta}_{2},\bm{\beta}^{\prime}}\left(\bm{X},Y\right)^{2}\right|\le C\left(\bm{X},Y\right)\left\Vert \bm{\beta}_{1}-\bm{\beta}_{2}\right\Vert \text{,}
\]
for each $\bm{\beta}_{1},\bm{\beta}_{2},\bm{\beta}^{\prime}\in\Theta$,
where $\mathbb{E}\left[C\left(\bm{X},Y\right)\right]<\infty$ and
if $\text{\ensuremath{l_{\bm{\beta},\bm{\beta}^{\prime}}\left(\bm{X},Y\right)^{2}}}$
is differentiable for each $\left(\bm{X},Y\right)$ (cf. \citealt[Thm. 9.56]{shapiro2021lectures}).
Differentiability follows since $\text{\ensuremath{l_{\bm{\beta},\bm{\beta}}\left(\bm{X},Y\right)^{2}}}$
is quadratic in $\bm{\beta}$. To verify the Lipschitz condition,
we follow the same approach as when checking the $\mathbb{P}$-Donsker
property and write
\[
\text{\ensuremath{l_{\bm{\beta},\bm{\beta}^{\prime}}\left(\bm{X},Y\right)^{2}}}=\left\{ \sum_{j=1}^{p}X_{j}\beta_{0j}-\sum_{j=1}^{p}X_{j}t_{k}\left(\beta_{j}^{\prime}\right)\beta_{j}+\epsilon\right\} ^{4}\text{,}
\]
where the derivative with respect to $\bm{\beta}$ is
\[
\nabla_{1}\left[\text{\ensuremath{l_{\bm{\beta},\bm{\beta}^{\prime}}\left(\bm{X},Y\right)^{2}}}\right]=\left[4X_{j}t_{k}\left(\beta_{j}^{\prime}\right)\left(\sum_{j=1}^{p}X_{j}\beta_{0j}-\sum_{j=1}^{p}X_{j}t_{k}\left(\beta_{j}^{\prime}\right)\beta_{j}+\epsilon\right)^{3}\right]_{j\in\left[p\right]}\text{,}
\]
therefore, since $\Theta$ is compact, and the highest power of $X_{j}$
and $\epsilon$ is $4$, the extreme value theorem and Holder's inequality
then yields a positive function $a:\mathbb{R}^{2+p}\to\mathbb{R}_{\ge0}$
that is linear in its first two inputs, such that
\[
\left\Vert \nabla_{1}\left[\text{\ensuremath{l_{\bm{\beta},\bm{\beta}^{\prime}}\left(\bm{X},Y\right)^{2}}}\right]\right\Vert _{1}\le a\left(\left\Vert \bm{X}\right\Vert ^{4},\epsilon^{4},\bm{\beta}^{\prime}\right)\text{.}
\]
The mean value theorem then suggests that we can take $C\left(\bm{X},Y\right)=C\times a\left(\left\Vert \bm{X}\right\Vert ^{4},\epsilon^{4},\bm{\beta}^{\prime}\right)$,
for some $C<\infty$, where $\mathbb{E}\left[a\left(\left\Vert \bm{X}\right\Vert _{}^{4},\epsilon^{4},\bm{\beta}^{\prime}\right)\right]<\infty$,
as required, by Assumption \ref{assu:fourth=000020moments}. \citet[Thm. 9.56]{shapiro2021lectures}
then implies that 
\begin{align}
\nabla_{1}P\left[l_{\bm{\beta},\bm{\beta}^{\prime}}\left(\bm{X},Y\right)^{2}\right] & =P\left[\nabla_{1}\left\{ l_{\bm{\beta},\bm{\beta}^{\prime}}\left(\bm{X},Y\right)^{2}\right\} \right]\nonumber \\
 & =\left[4P\left\{ X_{j}t_{k}\left(\beta_{j}^{\prime}\right)\left(\sum_{j=1}^{p}X_{j}\beta_{0j}-\sum_{j=1}^{p}X_{j}t_{k}\left(\beta_{j}^{\prime}\right)\beta_{j}+\epsilon\right)^{3}\right\} \right]_{j\in\left[p\right]}\text{.}\label{eq:=000020nabla=000020P=000020l2}
\end{align}
The extreme value theorem and Holder's inequality again yield a positive
continuous function $a:\mathbb{R}^{2+p}\to\mathbb{R}_{\ge0}$ , such
that
\[
\left\Vert \nabla_{1}P\left[l_{\bm{\beta},\bm{\beta}^{\prime}}\left(\bm{X},Y\right)^{2}\right]\right\Vert _{1}\le a\left(\mathbb{E}\left\Vert \bm{X}\right\Vert ^{4},\mathbb{E}\left[\epsilon^{4}\right],\bm{\beta}^{\prime}\right)\text{,}
\]
where the continuity is due to Assumption \ref{assu:=000020thresholds}.
By the mean value theorem and the compactness of $\Theta$, we can
choose
\[
\bar{D}=C\times\sup_{\bm{\beta}^{\prime}\in\Theta}a\left(\mathbb{E}\left\Vert \bm{X}\right\Vert ^{4},\mathbb{E}\left[\epsilon^{4}\right],\bm{\beta}^{\prime}\right)<\infty\text{,}
\]
such that 
\[
\text{(2A)},\text{(2C)}\le\bar{D}\left\Vert \bm{\beta}-\bm{\beta}^{\prime}\right\Vert \text{,}
\]
for each $\bm{\beta},\bm{\beta}^{\prime}\in\Theta$, as required.
The Lipschitz condition for (2B) follows similarly, thus Lemma \ref{lem:Equicont}
yields the equicontinuity of $\left(\Delta_{n}\right)$ and therefore
$\bar{\Delta}_{n}\overset{\mathbb{P}}{\longrightarrow}0$ follows
from Lemma \ref{lem:If=000020SE=000020then=000020uniform}, thus proving
Proposition \ref{prop:1A}.

Next, we verify the weak convergence of (1B). Indeed, since ${\cal L}$
is $\mathbb{P}$-Donsker, and
\[
{\cal L}_{\bm{\beta}_{0}}=\left\{ l_{\bm{\beta},\bm{\beta}_{0}}\left(\bm{X},Y\right)=\left\{ Y-\bm{X}^{\top}\bm{T}_{k}\left(\bm{\beta};\bm{\beta}_{0}\right)\right\} ^{2}:\bm{\beta}\in\Theta\right\} \subset{\cal L}\text{,}
\]
we have that ${\cal L}_{\bm{\beta}_{0}}$ is also $\mathbb{P}$-Donsker
by \citet[Thm. 2.10.1]{van2023weak}. Thus, by definition of a $\mathbb{P}$-Donsker
class (cf. \citealp[Sec. 2.1]{van2023weak}), we have our required
result.
\begin{prop}
Under Assumptions \ref{assu:=000020thresholds} and \ref{assu:fourth=000020moments},
$\text{(1B)}\rightsquigarrow R$, where $R:\Omega\to\ell^{\infty}\left(\Theta\right)$.
\end{prop}
It remains to show that $\text{(1C)}\rightsquigarrow\bm{Z}^{\top}\nabla_{2}g(\cdot,\bm{\beta}_0)$.
To this end, we follow the suggestion from \citet[Sec. 3.13]{van2023weak}
and apply the delta method, to the function $\bm{\beta}^{\prime}\mapsto g\left(\cdot;\bm{\beta}^{\prime}\right)\in\ell^{\infty}\left(\Theta\right)$.
We will again use the directional functional delta method from \citet[Fact 3.2]{westerhout2024asymptotic}. 

We first obtain the Hadamard directional derivative of $g\left(\cdot;\bm{\beta}^{\prime}\right)$
in $\bm{\beta}^{\prime}$ and arbitrary direction $\bm{\eta}\in\mathbb{R}^{p}$.
We note that it is easy to check that for any compact set $\Theta^{\prime}\subset\mathbb{R}^{p}$,
the Lipschitz condition 
\[
\left\Vert g\left(\cdot;\bm{\beta}\right)-g\left(\cdot;\bm{\beta}^{\prime}\right)\right\Vert _{\infty}\le C\left\Vert \bm{\beta}-\bm{\beta}^{\prime}\right\Vert \text{,}
\]
is satisfied, for some $C>0$ and every $\bm{\beta},\bm{\beta}^{\prime}\in\Theta^{\prime}$.
Here the supremum norm is taken over the compact domain $\Theta$.
As such, $g\left(\cdot;\bm{\beta}^{\prime}\right)$ is locally Lipschitz
and the Hadamard and Gateau derivatives coincide (cf. \citealp[Prop. 5.9]{penot2016analysis}).

For brevity, we write
\[
\bm{t}\left(\bm{\beta}\right)=\left[t\left(\beta_{j}\right)\right]_{j\in\left[p\right]}\text{,}
\]
where $t=t_{k}$, for the same fixed $k$ as in $g$. Under Assumptions
\ref{assu:=000020thresholds} and \ref{assu:fourth=000020moments},
we have
\[
\nabla_{2}g\left(\bm{\beta};\bm{\beta}^{\prime}\right)=2\left[\mathbb{E}Y-\left\{ \mathbb{E}\bm{X}\right\} ^{\top}\mathbf{D}\left(\bm{\beta}^{\prime}\right)\bm{\beta}\right]\nabla\bm{t}\left(\bm{\beta}^{\prime}\right)\text{,}
\]
for each $\bm{\beta}\in\Theta$, where the derivative is taken with
respect to $\bm{\beta}^{\prime}$. By definition, $\nabla_{2}g\left(\cdot;\bm{\beta}^{\prime}\right)^{\top}\bm{\eta}$
is the Gateau, and thus Hadamard derivative in the direction $\bm{\eta}\in\mathbb{R}^{p}$
if
\[
\lim_{s\searrow0}\left\Vert \frac{g\left(\cdot;\bm{\beta}^{\prime}+s\bm{\eta}\right)-g\left(\cdot;\bm{\beta}^{\prime}\right)}{s}-\nabla g\left(\cdot;\bm{\beta}^{\prime}\right)^{\top}\bm{\eta}\right\Vert _{\infty}=0\text{.}
\]
By Taylor's theorem, for each $\bm{\beta},\bm{\beta}^{\prime}\in\Theta$,
it holds that 
\begin{equation}
\frac{g\left(\bm{\beta};\bm{\beta}^{\prime}+s\bm{\eta}\right)-g\left(\bm{\beta};\bm{\beta}^{\prime}\right)}{s}-\nabla g\left(\bm{\beta};\bm{\beta}^{\prime}\right)^{\top}\bm{\eta}=s\bm{\eta}^{\top}\mathbf{H}\left(\bm{\beta};\bar{\bm{\beta}}\right)\bm{\eta}\text{,}\label{eq:=000020Taylor=000020theorem}
\end{equation}
for some $\bar{\bm{\beta}}\in\Theta$, where $\mathbf{H}\left(\bm{\beta};\bar{\bm{\beta}}\right)$
is the Hessian of $g\left(\bm{\beta};\bm{\beta}^{\prime}\right)$
in $\bm{\beta}^{\prime}$, where the righthand side of (\ref{eq:=000020Taylor=000020theorem})
is continuous in both $\left(\bm{\beta},\bar{\bm{\beta}}\right)\in\Theta\times\Theta$,
for each $s$ and $\bm{\eta}$, by Assumptions \ref{assu:=000020thresholds}.
Thus,
\[
\left\Vert \frac{g\left(\cdot;\bm{\beta}^{\prime}+s\bm{\eta}\right)-g\left(\cdot;\bm{\beta}^{\prime}\right)}{s}-\nabla g\left(\cdot;\bm{\beta}^{\prime}\right)^{\top}\bm{\eta}\right\Vert _{\infty}\le s\sup_{\bar{\bm{\beta}}\in\Theta}\left\Vert \bm{\eta}^{\top}\mathbf{H}\left(\cdot;\bar{\bm{\beta}}\right)\bm{\eta}\right\Vert _{\infty}\to0
\]
as $s\searrow0$, by the extreme value theorem as required. Since
$\sqrt{n}\left(\hat{\bm{\beta}}_{n}-\bm{\beta}_{0}\right)\rightsquigarrow\bm{Z}$,
from Assumption \ref{assu:asymptotic=000020beta}, the delta method
of \citet[Fact 3.2]{westerhout2024asymptotic} then yields the following
result, which completes the proof.
\begin{prop}
Under Assumptions \ref{assu:=000020thresholds}, \ref{assu:fourth=000020moments},
and \ref{assu:asymptotic=000020beta}, it holds that $\text{(1C)}\rightsquigarrow\bm{Z}^{\top}\nabla_{2}g(\cdot,\bm{\beta}_0)$.
\end{prop}

\section{Proof of Theorem 4}\label{sec:Proof=000020of=000020theorem=0000204}
We first prove that $\mathbb{P}(\bar{\bm{\beta}}_{n,\mathrm I} = \bm{0}) \to 1$, as $n\to\infty$. By the definition of $\bar{\bm{\beta}}_{n}$ and the union bound,
\begin{align*}
    \mathbb{P}(\bar{\bm{\beta}}_{n,\mathrm I} \ne \bm{0}) &= \mathbb{P}\left(\bigcup_{j \in {\cal S}_{0}}\left\{ \mathbf{1}\left(t_{\hat K_n}\left(\hat{\beta}_{n,j}\right)=0\right)\ne 0\right\} \right)\\
    &=\mathbb{P}\left(\bigcup_{j \in {\cal S}_{0}}\left\{|\hat{\beta}_{n,j}| > \delta_{\hat K_n}\right\} \right)\\
    &\leq \sum_{j \in {\cal S}_{0}}\mathbb{P}\left(|\hat{\beta}_{n,j}| > \delta_{\hat K_n} \right).
\end{align*}
Then for each $j \in {\cal S}_0$, by the consistency of $\hat K_n$ and $\hat {\bm\beta}_n$, we have that
\begin{align*}
 \mathbb{P}\left(|\hat{\beta}_{n,j}| > \delta_{\hat K_n} \right) &= \sum_{k \ne k_0} \left\{\mathbb{P}\left(|\hat{\beta}_{n,j}| > \delta_{k}, \hat K_n = k \right)\right\} + \mathbb{P}\left(|\hat{\beta}_{n,j}| > \delta_{k_0}, \hat K_n = k_0 \right)\\
 &\leq \sum_{k \ne k_0} \left\{\mathbb{P}\left( \hat K_n = k \right)\right\} + \mathbb{P}\left(|\hat{\beta}_{n,j}| > \delta_{k_0} \right)
 \longrightarrow 0.
\end{align*}
As $\left|{\cal S}_{0}\right|<\infty$, this completes the proof regarding the consistency on the irrelevant set. For the relevant set, we write \[\sqrt{n}(\bar{\bm\beta}_{n,\mathrm R} - \bm\beta_{0,\mathrm R}) = \sqrt{n}(\bar{\bm\beta}_{n,\mathrm R} - \hat{\bm\beta}_{n,\mathrm R}) + \sqrt{n}(\hat{\bm\beta}_{n,\mathrm R} - \bm\beta_{0,\mathrm R}).\]
The second term on the RHS converges weakly to $\bm{Z}_{\mathrm R}$ by Assumption \ref{assu:asymptotic=000020beta}, and so, by Slutsky's theorem, it suffices to show that
$$\lim_{n\to\infty}\mathbb{P}\left(\sqrt{n}(\bar{\bm\beta}_{n,\mathrm R} - \hat{\bm\beta}_{n,\mathrm R}) = \bm{0}\right) = 1.$$ Similarly to the irrelevant case, the definition of $\bar{\bm\beta}_n$ and the union bound yields
\begin{align*}
    \mathbb{P}\left(\sqrt{n}(\bar{\bm\beta}_{n,\mathrm R} - \hat{\bm\beta}_{n,\mathrm R}) \ne \bm{0}\right) &= \mathbb{P}\left(\bigcup_{j \in {\cal S}_{0}^{\text{c}}}\left\{|\hat{\beta}_{n,j}| \leq \delta_{\hat K_n}\right\} \right)\\
    &\leq \sum_{j \in {\cal S}_{0}^{\text{c}}}\mathbb{P}\left(|\hat{\beta}_{n,j}| \leq \delta_{\hat K_n} \right).
\end{align*}
For each $j \in {\cal S}_0^{\text{c}}$, an identical argument to the irrelevant case yields$$\lim_{n\to\infty}\mathbb{P}\left(|\hat{\beta}_{n,j}| \leq \delta_{\hat K_n} \right) = 0,$$ which completes the proof.

\bibliographystyle{plainnat}
\bibliography{biblio}

\end{document}